\documentclass[12pt]{amsart}
\usepackage{amsmath,amsfonts,amssymb,amsthm,amstext,pgf,graphicx,hyperref,verbatim,ytableau,textcomp,color,young,youngtab,tikz,tikz-cd,bbold,mathabx}
\usepackage{cleveref}
\usepackage[mathscr]{euscript}
\definecolor{asparagus}{rgb}{0.0, 0.5, 0.0}
\theoremstyle{plain}
\newtheorem{thm}{Theorem}[section]
\newtheorem{lem}[thm]{Lemma}
\newtheorem{prop}[thm]{Proposition}
\newtheorem{cor}[thm]{Corollary}
\newtheorem{rem}[thm]{Remark}

\newtheorem{exmp}[thm]{Example}

\theoremstyle{definition}
\newtheorem{defn}{Definition}[section]
\newtheorem{conj}{Conjecture}[section]
\newcommand{\f}{\mathfrak{f}}
\newcommand{\g}{\mathfrak{g}}
\newcommand{\h}{\mathfrak{h}}
\DeclareMathOperator{\ST}{Star_n}
\DeclareMathOperator{\Gid}{id_G}

\DeclareMathOperator{\rel}{rel}
\DeclareMathOperator{\cov}{Cov}
\DeclareMathOperator{\var}{Var}
\DeclareMathOperator{\trivial}{triv}
\DeclareMathOperator{\Tr}{trace}
\DeclareMathOperator{\Par}{Par}

\DeclareMathOperator{\std}{Std}

\DeclareMathOperator{\G}{\mathscr{G}}
\title{}
\begin{document}
\title[Noise Sensitivity Governed by Continuous-Time Random Walks on $S_n$]{Noise Sensitivity Governed by Continuous-Time Random Walks on the Symmetric Group}
%\title[Noise Sensitivity of Long Cycles on the Discrete $d$-Torus]{Noise Sensitivity of Long Cycles in the Interchange Process on the discrete $d$-Torus}
\author{Gideon Amir}
\address[Gideon Amir]{Department of Mathematics, Bar-Ilan University, Ramat-Gan 5290002, Israel.}
\email{gidi.amir@gmail.com}
\author{Subhajit Ghosh}
\address[Subhajit Ghosh]{Department of Mathematics, Indian Institute of Technology Madras, Chennai-600 036, India.}
\email{gsubhajit@alum.iisc.ac.in}
\makeatletter
\@namedef{subjclassname@2020}{%
	\textup{2020} Mathematics Subject Classification}
\makeatother
%----------Abstract--------------
\begin{abstract}
	We study the noise sensitivity of Boolean functions on the symmetric group $S_n$, where noise is induced by running a continuous-time Markov chain on $S_n$. We investigate the noise sensitivity of several fundamental functions, including the parity function and analogues of the dictator function. We focus, in particular, on the case where the underlying Markov chain is an interchange process on the complete graph. We establish comparison results between noise sensitivity with respect to the interchange process on the complete graph and the weighted interchange process on any connected graph. We also show that the indicator function of long cycles is noise sensitive under the interchange process on each of these graphs. Finally, using the fact that the interchange process on the complete graph $K_n$ is the continuous-time random walk generated by all transpositions, we prove that noise sensitivity is preserved when the noise source is changed to the continuous-time random walk generated by all $s$-cycles, where $s$ is even and $2<s\ll n$.
\end{abstract}
\maketitle
\section{Introduction}\label{intro}
In many scientific disciplines involving random or non-random processes, a system takes an input and returns an output. A natural question then arises: How much do small changes in the input affect the output? This question is one of the key concerns for modern-day scientists (see \cite{K_surcvey_on_NS_of_universe} and references therein). It is mathematically modeled and extensively studied within the framework of noise sensitivity. 

Noise sensitivity has been extensively studied across diverse fields, including probability, combinatorics, mathematical physics, theoretical computer science, and social choice theory (see \cite{Kalai_introduction1,MOO,SS,ST,T_introduction,TV_introduction} for details). The concept applies to both classical and quantum stochastic models (cf. \cite{T_classical_and_quantum}). Following the seminal work of Benjamini, Kalai, and Schramm \cite{BKS} in the late twentieth century, the theory of noise sensitivity developed into a well-defined and significant area of study in probability theory and statistical physics. In their work, they analyzed the noise sensitivity of Boolean functions defined on the \emph{Hamming cube}, $\{0,1\}^n$. They applied the theory to problems in bond percolation -- demonstrating, in particular, that the crossing event for bond percolation on a rectangular grid is noise sensitive.

A sequence of Boolean functions defined on the vertices of the \emph{Hamming cube}, i.e., $\{0,1\}^n$, is said to be noise sensitive if it asymptotically loses correlation when a positive fraction of the input bits are affected by some noise. Here, noising a bit means refreshing it with some positive probability. Formally, Benjamini, Kalai, and Schramm's definition is as follows:
\begin{defn}\label{def:BKS_NS+NStable}
	Let $n\geq1$ be a positive integer, and $\f_n:\{0,1\}^n\rightarrow\{0,1\}$ be a Boolean function indexed by $n$. Also, let $X_1,\dots,X_n$ be independent Bernoulli$(\frac{1}{2})$ random variables. Given $0<\varepsilon<1$, the sequence of Boolean functions $\{\f_n\}_n$ is asymptotically 
	\begin{itemize}
		\item \emph{noise sensitive} if \quad
		$\displaystyle\lim_{n\rightarrow\infty}\cov(\f_n(X_1,\dots,X_n),\f_n(X_1(\varepsilon),\dots,X_n(\varepsilon)))=0,\quad\text{ and }$
		\item \emph{noise stable} if \quad
		$\displaystyle\lim_{\varepsilon\rightarrow0}\sup_n\mathbb{P}(\f_n(X_1,\dots,X_n)\neq \f_n(X_1(\varepsilon),\dots,X_n(\varepsilon)))=0,$
	\end{itemize}
	where, 
	\[X_i(\varepsilon)=\begin{cases}
		X_i&\text{ with probability }1-\varepsilon,\\
		1-X_i&\text{ with probability }\varepsilon,
	\end{cases}\text{ for all }1\leq i\leq n.\]
\end{defn}
The functions in the sequence $\{\f_n\}_n$ are defined on $\{0,1\}^n$, where inputs are uniformly chosen vertices of the Hamming cube. The noise source can be modeled by a continuous-time nearest-neighbor random walk on the hypercube. A single step of this random walk corresponds to flipping one input bit. More formally, there are Poisson clocks (alarm clocks that ring at times distributed according to the exponential distribution) at each coordinate, and these clocks ring independently with rate $1$. The process flips a coordinate whenever the corresponding clock rings. The perturbed input is then the position of the aforementioned continuous-time random walk at time $\alpha$, where $\varepsilon=\frac{1}{2}(1-e^{-2\alpha})$. This motivates the definition of noise sensitivity for a function on the state space of a Markov chain. Informally, it says that noise sensitivity means that the function decorrelates faster than the relaxation time of the entire system. The formal definition is given below.
 \begin{defn}\label{def:Spectral_NS+NStable}
 	Let $\{(X_n(t),\Omega_n,\pi_n)\}_n$ be a sequence of reversible and irreducible continuous-time Markov chains. Here $\Omega_n$ is the state space, and $\pi_n$ is the stationary distribution of the $n$th chain $X_n(t)$. Also, let $t_{\rel}^n$ be the \emph{relaxation time} (will be defined) of the $n$th chain $X_n(t)$. A sequence $\{\f_n\}_n$ of Boolean functions,
 	$\f_n : \Omega_n \longrightarrow \{0, 1\}$, is said to be
 	\begin{itemize}
 		\item \emph{Noise sensitive} with respect to $(X_n(t),\Omega_n,\pi_n)$ if for all $\alpha>0$,
 		\begin{equation}\label{eq:Spectral_NS-def}
 			\lim_{n\rightarrow\infty}\cov\left(\f_n(X_n(0)),\;\f_n(X_n(\alpha t^n_{\rel}))\right)=0,\text{ where }X_n(0)\sim \pi_n.
 		\end{equation}
 		\item \emph{Noise stable} with respect to $(X_n(t),\Omega_n,\pi_n)$ if
 		\begin{equation}\label{eq:Spectral_NStable-def}
 			\lim_{\alpha\rightarrow0}\sup_n\mathbb{P}\left(\f_n(X_n(0))\neq \f_n(X_n(\alpha t^n_{\rel}))\right)=0,\text{ where }X_n(0)\sim \pi_n.
 		\end{equation}
 	\end{itemize}
 \end{defn}
Both \Cref{def:BKS_NS+NStable} and \Cref{def:Spectral_NS+NStable} coincide, as the relaxation time for the continuous-time nearest-neighbor random walk on the Hamming cube is $\frac{1}{2}$. This notion of noise sensitivity and stability, where the noise is governed by a continuous-time Markov chain, first arose in the work of Broman et al. \cite{BGS} and Forsstr\"{o}m \cite{Forsstrom_EJP}. Both Broman et al. \cite{BGS} and Forsstr\"{o}m \cite{Forsstrom_EJP} considered the exclusion process on a graph and referred to it as \emph{exclusion sensitivity} and \emph{exclusion stability}. The general definition, \Cref{def:Spectral_NS+NStable}, can be found in \cite[Definition 1.1 and Definition 1.2]{Forsstrom_arXiv}.
 
 The idea is that working with very general chains is too broad to say much, so we decide to focus on chains on the symmetric group, where a rich representation theory exists. In this paper, we work with \Cref{def:Spectral_NS+NStable}, and our noise sources will be \emph{interchange processes} (which will be explained in the next paragraph) on certain graphs. More precisely, the continuous-time Markov chains described in \Cref{def:Spectral_NS+NStable} correspond to continuous-time random walks on the symmetric group. %We now proceed to define the interchange process.
  
  Let $\mathscr{G}=(V,E)$ be any graph with vertex set $V=[n]:=\{1,\dots,n\}$. The \emph{weighted interchange process} on $\mathscr{G}$ is defined as follows: Suppose the edges of the graph $\mathscr{G}$ are equipped with Poisson clocks, with all clocks being independent. Let the edge $\{u,v\}\in E$ be equipped with a Poisson clock of rate $x_{u,v}$. Now place $n$ distinct marbles on the vertices of the graph. Whenever the clock at $\{u,v\}$ rings, exchange the marbles at vertices $u$ and $v$. Thus, each marble performs a continuous-time random walk on the graph, though the walks of different marbles are dependent. The process can be viewed as a continuous-time random walk on the symmetric group $S_n$, since the positions of the marbles at any given time correspond to a permutation of their original positions. \emph{We simply call it the interchange process when all nonzero weights $x_{u,v}$ for $\{u,v\} \in E$ are equal.} 
   %%%%%%%%%%-------K_6---ST_6------%%%%%%%%%%%%
  \begin{figure}[ht]
  	\centering
  	\begin{tikzpicture}[scale=0.80,
  		vertex/.style={circle,fill=black,inner sep=2.5pt}]
  		%--------------------------------------------------
  		% Complete graph K_6
  		%--------------------------------------------------
  		\begin{scope}[xshift=0cm]
  			\node[vertex,label=above:{$1$}] (v1) at (90:2) {};
  			\node[vertex,label=above:{$2$}] (v2) at (30:2) {};
  			\node[vertex,label=right:{$3$}] (v3) at (-30:2) {};
  			\node[vertex,label=below:{$4$}] (v4) at (-90:2) {};
  			\node[vertex,label=below:{$5$}] (v5) at (-150:2) {};
  			\node[vertex,label=left:{$6$}] (v6) at (150:2) {};
  			\foreach \i in {1,...,6}
  			{
  				\foreach \j in {\i,...,6}
  				{
  					\ifnum\i<\j
  					\draw (v\i)--(v\j);
  					\fi
  				}
  			}
  		\end{scope}
  		%--------------------------------------------------
  		% Star graph ST_6
  		%--------------------------------------------------
  		\begin{scope}[xshift=8cm]
  			\node[vertex,label=left:{$6$}] (c) at (0,0) {};
  			\node[vertex,label=above:{$1$}] (s1) at (-1.4,1.6) {};
  			\node[vertex,label=above:{$2$}] (s2) at (1.4,1.6) {};
  			\node[vertex,label=right:{$3$}] (s3) at (2.0,-0.6) {};
  			\node[vertex,label=below:{$4$}] (s4) at (0.5,-2.0) {};
  			\node[vertex,label=left:{$5$}] (s5) at (-1.8,-1.0) {};
  			\foreach \x in {s1,s2,s3,s4,s5}
  			\draw (c)--(\x);
  		\end{scope}
  	\end{tikzpicture}
  	\caption{Complete graph (left) and star graph (right) on $6$ vertices.}\label{fig:K6ST6}
  \end{figure}
  
  %%%%%%%%%%%%%%%%%%%%%%%%%%%%%%%%%%
  We establish a comparison result for noise sensitivity with respect to the interchange process on the complete graph and the weighted interchange process on an arbitrary connected graph. We begin with the mean-field (highly homogeneous) setting corresponding to the interchange process on the complete graph on $n$ vertices. Recall that the complete graph on $n$ vertices, denoted by $K_n$, is the undirected graph with vertex set $[n]$ and edge set $\{\{i,j\}: 1 \leq i < j \leq n\}$ (see Figure \ref{fig:K6ST6} (left)). We then relate the mean-field setting to the weighted interchange process on any connected graph on $n$ vertices. Our comparison theorem is as follows:
  \begin{thm}\label{thm:gen<->complete}
  	For every $n\geq 1$, let $\f_n\rightarrow \{0,1\}$ be a Boolean function, and let $\mathscr{G}$ be a connected graph on $n$ vertices. Then the following statements hold:
  	\begin{enumerate}
  		\item If the sequence $\{\f_n\}_{n=1}^{\infty}$ is noise sensitive with respect to the interchange process on $K_n$, then it is noise sensitive with respect to the weighted interchange process on $\mathscr{G}$.
  		\item If the sequence $\{\f_n\}_{n=1}^{\infty}$ is noise stable with respect to the interchange process on $\mathscr{G}$, then it is noise stable with respect to the interchange process on $K_n$.
  	\end{enumerate}
  \end{thm} 
Recall that the star graph on $n$ vertices, denoted by $\ST$, is the undirected graph with vertex set $[n]$ and edge set $\{\{i,n\}: 1 \leq i < n\}$ (see Figure \ref{fig:K6ST6} (right)). The \emph{dictator functions} given in Examples \ref{ex:dictator} and \ref{ex:dictator-star}, considered in the context of the interchange process on $\ST$, show that the converse of Theorem \ref{thm:gen<->complete} does not hold in general.
%%%%%%%%%
%%%%%%%%%%
%%%%%%%%%%%

Since the interchange process on $K_n$ is the continuous-time random walk on the symmetric group $S_n$ generated by all transpositions, we further generalize this setting and prove the following universality property.
  \begin{thm}\label{thm:RT<-->s-cycle}
  	Let $s$ be an even integer satisfying $2<s\ll n$. Then a Boolean function is noise sensitive (respectively stable) with respect to the interchange process on $K_n$ if and only if it is noise sensitive (respectively stable) with respect to the continuous-time random walk on $S_n$ generated by all $s$-cycles.
  \end{thm}
  An important example of a sequence of Boolean functions is given by the indicator of the set of permutations containing a `long' cycle. The precise result is stated below.
\begin{thm}\label{thm:large_cycle1}
	Let $0< c<1$ be fixed. Then the function $\xi_n:S_n\rightarrow \{0,1\}$ defined by
	\[\xi_n:=\mathbb{1}_{\{\pi\in S_n:\text{the length of the largest cycle in }\pi\text{ is at least }cn\}},\]
	is noise sensitive with respect to the interchange process on $K_n$.
\end{thm}
Theorem \ref{thm:large_cycle1} guarantees that the sequence $\{\xi_n\}_{n=1}^{\infty}$ is noise sensitive with respect to:
\begin{itemize}
	\item the interchange process on $\mathscr{G}$, by Theorem \ref{thm:gen<->complete}
	\item the continuous-time random walk on $S_n$ generated by all $s$-cycles, by Theorem \ref{thm:RT<-->s-cycle}.
\end{itemize}
In particular, for a fixed positive integer $d$, taking $\mathscr{G}$ to be the $d$-dimensional integer lattice $\{-n,\dots,-1,1,\dots,n\}^d$, Theorems \ref{thm:gen<->complete} and Theorem \ref{thm:large_cycle1} further imply the following corollary.
\begin{cor}\label{cor:long_torus}
  For a fixed positive integer $d$, the sequence of functions $\{\xi_{(2n)^d}\}_{n=1}^{\infty}$	is noise sensitive with respect to the interchange process on the $d$-dimensional integer lattice $\{-n,\dots,-1,1,\dots,n\}^d$.
  \end{cor}
  Our motivation for considering the Boolean functions appearing in Corollary \ref{cor:long_torus} stems from a celebrated result of B\'{a}lint T\'{o}th on the \emph{quantum Heisenberg ferromagnet} \cite{Toth-conjecture}. T\'{o}th's representation formula for the spontaneous magnetization of the quantum Heisenberg ferromagnet at inverse temperature $\beta$ is closely related to the interchange process on the three-dimensional integer lattice $\{-n, \dots, -1, 1, \dots, n\}^3$, and he conjectured that the spontaneous magnetization undergoes a phase transition. It is therefore natural to formulate an analogous conjecture for the interchange process on the lattice $\{-n, \dots, -1, 1, \dots, n\}^3$. T\'{o}th himself suggested that this could serve as an interesting toy model. The conjecture is as follows.
  \begin{conj}[{\cite[Section 5: T\'{o}th's conjecture]{AK}}]\label{conj:Toth}
  	Let $\sigma_\beta$ denote the permutation generated by the interchange process on the lattice $\{-n,\dots,-1,1,\dots,n\}^3$ at time $\beta$. Then, for every fixed $c\in(0,1)$,
  	\[\lim_{n\rightarrow\infty}\mathbb{P}(\text{the length of the largest cycle of }\sigma_{\beta}>c (2n)^3)\]
  	undergoes a phase transition as a function of $\beta$; it is zero for $\beta<\beta'$ (critical inverse temperature) and positive for $\beta>\beta'$
  \end{conj}
A significant recent advance towards Conjecture \ref{conj:Toth}is due to Dor Elboim and Allan Sly \cite{Sly_et_al}, who proved that, \emph{for all $d \geq 5$ and all sufficiently large $\beta$, the permutation $\sigma_{\beta}$ contains infinite cycles almost surely.}
  
 Another source of motivation is the work of Alon and Kozma \cite{AK}, who expressed the number of cycles of a given length in a permutation as a linear combination of irreducible characters of the symmetric group. Using this formula, they obtained an explicit expression for the probability that the permutation generated by the interchange process on a connected graph is an $n$-cycle at time $t$, together with estimates for the probabilities of observing shorter cycles. For further details, we refer the reader to \cite{AK} and the references therein.
 
\emph{Organisation of the paper.} In Section \ref{sec:preliminaries}, we introduce the notation and preliminaries for the general setup. Section \ref{sec:examples} presents several examples involving fundamental functions, such as the parity function and analogues of the dictator function. In Section \ref{sec:Sn_repn}, we discuss the representation theory of the symmetric group and related combinatorics. In Section \ref{sec:proof_of_thm:1.1}, we establish a (non-commutative) Fourier-analytic criterion and compare the interchange processes on $K_n$ and $\mathscr{G}$, leading to Theorem \ref{thm:gen<->complete}. Section \ref{sec:proof_of_thm:1.3} is devoted to continuous-time random walks on $S_n$ generated by cycles and the proof of Theorem \ref{thm:RT<-->s-cycle}. Finally, in Section \ref{sec:proof_of_thm:1.4}, we prove Theorem \ref{thm:large_cycle1}.
\section{Notations and preliminaries}\label{sec:preliminaries}
Our methodology mainly uses the non-commutative Fourier analysis and the representation theory of the finite group, particularly $S_n$. In this section, we briefly recall the useful concepts from continuous-time Markov chain, the finite group representation, and the non-commutative Fourier analysis on finite groups. We refer the reader to the classical texts \cite{D1,Sagan,Serre} for more details.
\subsection{Boolean functions and continuous-time Markov chains}\label{subsec:CTMC_bkground}
We recall some standard facts and notation for Boolean functions defined on the state space of a continuous-time Markov chain to keep the article self-contained. For more details, we refer the reader to \cite{Forsstrom_arXiv,GS}. Let $Q_n=(q_{n,ij})_{i,j\in\Omega_n}$ be the infinitesimal generator of the reversible and irreducible continuous-time Markov chain $(X_n(t),\Omega_n,\pi_n)$ with stationary distribution $\pi_n$. Throughout this section, we will assume $X_n(0)\sim \pi_n$. Given any two functions $\f,\g:\Omega_n\rightarrow\mathbb{C}$, we define a (complex) inner product by
\begin{equation}\label{eq:MCIP}
	\langle \f,\g\rangle:=\sum_{\omega\in\Omega_n}\f(\omega)\overline{\g(\omega)}\pi_n(\omega)=\mathbb{E}\left(\f(X_n(0))\cdot \overline{\g(X_n(0))}\right).
\end{equation}
Reversibility ensures the existence of an $\langle \;,\;\rangle$-orthonormal basis $\{\psi_{n,j}\}_{j=0}^{|\Omega_n|-1}$, consisting of the eigenfunctions of $-Q_n$, with corresponding eigenvalue $\{\theta_{n,j}\}_{j=0}^{|\Omega_n|-1}$. Without loss of generality, we may assume that the eigenvalues $\theta_{n,0},\dots,\theta_{n,|\Omega_n|-1}$ satisfies the following
\[0=\theta_{n,0}<\theta_{n,1}\leq\theta_{n,2}\leq\cdots\leq\theta_{n,|\Omega_n|-1}.\]
Here the strict inequality $\theta_{n,0}<\theta_{n,1}$ follows from irreducibility. The unique eigenfunction $\psi_{n,0}$ corresponding to the smallest eigenvalue $\theta_{n,0}$ is given by $\psi_{n,0}(\omega)=1\text{ for all }\omega\in\Omega_n$. 
Therefore, for any $\f:\Omega_n\rightarrow\mathbb{R}$, using the decomposition $\f=\displaystyle\sum_{j=0}^{|\Omega_n|-1}\langle \f,\psi_{n,j}\rangle\psi_{n,j}$, we have
\begin{equation}\label{eq:CTMC_E-V}
	\mathbb{E}(\f(X_n(0)))=\langle \f,\psi_{n,0}\rangle\;\text{ and }\;\var(\f(X_n(0)))=\langle \f,\f\rangle-\langle \f,\psi_{n,0}\rangle^2=\sum_{j=1}^{|\Omega_n|-1}|\langle \f,\psi_{n,j}\rangle|^2.
\end{equation}
%We also note that (cf. proof of Proposition 1.4 in \cite{Forsstrom_arXiv})
%\begin{equation}\label{eq:cov-ev_relation}
	%	\cov\left(f(X_0^{(n)}),f(X_{\alpha t^{(n)}_{\rel}}^{(n)})\right)=\sum_{j=1}^{|\Omega^{(n)}|-1}e^{-\alpha\theta^{(n)}_j/\theta^{(n)}_0}\langle f,\psi_j^{(n)}\rangle^2.
	%\end{equation}
The \emph{relaxation time}, denoted $t_{\rel}^{n}$, is given by
\[\frac{1}{t_{\rel}^{n}}:=\inf_{\f\not\equiv 0:\mathbb{E}(\f)=0}\frac{\langle-Q_n\f,\f\rangle}{\langle \f,\f\rangle}=\theta_{n,1}\]

We will now prove two standard results.
\begin{lem}\label{lem:cov-ns}
	Consider the continuous-time Markov chain $(X_n(t),\Omega_n,\pi_n)$ from above, and recall our assumption $X_n(0)\sim\pi_n$. Then, for any function $\f:\Omega_n\rightarrow \mathbb{R}$, we have the following
	\[\cov\left(\f(X_n(0)),\f(X_n(t))\right)=\displaystyle\sum_{j=1}^{|\Omega_n|-1}|\langle \f,\psi_{n,j}\rangle|^2\exp(-t\theta_{n,j}).\]
\end{lem}
\begin{proof}
	As $X_n(0)\sim\pi_n$, using $\mathbb{P}(X_n(t)=y|X_n(0)=x)=(\exp(t Q_n)\mathbb{1}_{\{y\}})(x)$, we have\begin{align}\label{eq:cov-ns1}
		\mathbb{E}\left(\f(X_n(0))\f(X_n(t))\right)&=\sum_{x,y\in \Omega_n}\f(x)\f(y)\mathbb{P}(X_n(t)=y|X_n(0)=x)\mathbb{P}(X_n(0)=x)\nonumber\\
		%&=\sum_{x,y\in \Omega_n}f_n(x)f_n(y)(\exp(t Q_n)\mathbb{1}_{\{y\}})(x)\pi_n(x)\nonumber\\
		&=\sum_{x\in \Omega_n}\pi_n(x)\f(x)\sum_{y\in \Omega_n}\f(y)(\exp(t Q_n)\mathbb{1}_{\{y\}})(x)\nonumber\\
		&=\sum_{x\in \Omega_n}\pi_n(x)\f(x)(\exp(t Q_n)\f)(x)
	\end{align}
	Now using the decomposition $\f=\displaystyle\sum_{j=0}^{|\Omega_n|-1}\langle \f,\psi_{n,j}\rangle\psi_{n,j}$, we obtain
%%%\begin{align*}
%%%		\exp(t Q_n)\f&=\displaystyle\sum_{j=0}^{|\Omega_n|-1}\langle \f,\psi_{n,j}\rangle\exp(t Q_n)\psi_{n,j}\\
%%%		&=\displaystyle\sum_{j=0}^{|\Omega_n|-1}\langle \f,\psi_{n,j}\rangle\exp(-t\theta_{n,j})\psi_{n,j}.
%%%	\end{align*}
\[\exp(t Q_n)\f=\displaystyle\sum_{j=0}^{|\Omega_n|-1}\langle \f,\psi_{n,j}\rangle\exp(t Q_n)\psi_{n,j}=\displaystyle\sum_{j=0}^{|\Omega_n|-1}\langle \f,\psi_{n,j}\rangle\exp(-t\theta_{n,j})\psi_{n,j}.\]
	Therefore, the expression in \eqref{eq:cov-ns1}, and hence $\mathbb{E}\left(\f(X_n(0))\f(X_n(t))\right)$ is equal to
	\begin{align*}
		&\sum_{x\in \Omega_n}\pi_n(x)\f(x)\displaystyle\sum_{j=0}^{|\Omega_n|-1}\langle \f,\psi_{n,j}\rangle\exp(-t\theta_{n,j})\psi_{n,j}(x)\\
		=&\displaystyle\sum_{j=0}^{|\Omega_n|-1}\langle \f,\psi_{n,j}\rangle\exp(-t\theta_{n,j})\sum_{x\in \Omega_n}\pi_n(x)\f(x)\psi_{n,j}(x)\\
		=&\displaystyle\sum_{j=0}^{|\Omega_n|-1}|\langle \f,\psi_{n,j}\rangle|^2\exp(-t\theta_{n,j}).
	\end{align*}
	We note that $X_n(t)\sim\pi_n$ as the initial distribution of the continuous-time Markov chain is the stationary distribution. Therefore, $\mathbb{E}\left(\f(X_n(0))\right)=\mathbb{E}\left(\f(X_n(t))\right)$, and hence,
	\begin{align*}
		\cov\left(\f(X_n(0)),\f(X_n(t))\right)&=\mathbb{E}\left(\f(X_n(0))\f(X_n(t))\right)-\mathbb{E}\left(\f(X_n(0))\right)\mathbb{E}\left(\f(X_n(t))\right)\nonumber\\
		&=\displaystyle\sum_{j=0}^{|\Omega_n|-1}|\langle \f,\psi_{n,j}\rangle|^2\exp(-t\theta_{n,j})-\left(\mathbb{E}\left(\f(X_n(0))\right)\right)^2\nonumber\\
		&=\displaystyle\sum_{j=0}^{|\Omega_n|-1}|\langle \f,\psi_{n,j}\rangle|^2\exp(-t\theta_{n,j})-\langle \f,\psi_{n,0}\rangle^2,\text{ by }\eqref{eq:CTMC_E-V}\nonumber\\
		&=\displaystyle\sum_{j=1}^{|\Omega_n|-1}|\langle \f,\psi_{n,j}\rangle|^2\exp(-t\theta_{n,j}),\quad\text{ using }\theta_{n,0}=0.\qedhere
	\end{align*}
\end{proof}
\begin{lem}\label{lem:boolean-st}
	Consider the continuous-time Markov chain $(X_n(t),\Omega_n,\pi_n)$ discussed above, and assuming $X_n(0)\sim\pi_n$. Then, for any Boolean function $\f:\Omega_n\rightarrow \{0,1\}$, we have
	\[\mathbb{P}\left(\f(X_n(0))\neq \f(X_n(t))\right)=2\displaystyle\sum_{j=1}^{|\Omega_n|-1}|\langle \f,\psi_{n,j}\rangle|^2\left(1-\exp(-t\theta_{n,j})\right).\]
\end{lem}
\begin{proof}
	Using $X_n(0)\sim\pi_n$, we obtain $\mathbb{E}\left(\f(X_n(0))\right)=\mathbb{E}\left(\f(X_n(t))\right)$. Also, the fact $\f$ is a Boolean function implies $\mathbb{E}\left(\f(X_n(0))\right)=\mathbb{E}\left(\f(X_n(0))^2\right)$, and
	\begin{align}\label{eq:boolean-st1}
		\mathbb{P}\left(\f(X_n(0))\neq \f(X_n(t))\right)&=\mathbb{E}\left[\f(X_n(t))(1-\f(X_n(0)))\right]+\mathbb{E}\left[(1-\f(X_n(t)))\f(X_n(0))\right]\nonumber\\
		&=2\left(\mathbb{E}\left[\f(X_n(0))\right]-\mathbb{E}\left[\f(X_n(0))\f(X_n(t))\right]\right)\nonumber\\
		&=2\left(\mathbb{E}\left[\f(X_n(0))^2\right]-\mathbb{E}\left[\f(X_n(0))\f(X_n(t))\right]\right)\nonumber\\
		&=2\left(\var\left(\f(X_n(0))\right)-\cov\left(\f(X_n(0)),\f(X_n(t))\right)\right)\nonumber\\
		&=2\left(\displaystyle\sum_{j=1}^{|\Omega_n|-1}|\langle \f,\psi_{n,j}\rangle|^2-\displaystyle\sum_{j=1}^{|\Omega_n|-1}|\langle \f,\psi_{n,j}\rangle|^2\exp(-t\theta_{n,j})\right)
	\end{align}
The equality in \eqref{eq:boolean-st1} follows from \eqref{eq:CTMC_E-V}, and Lemma \ref{lem:cov-ns}. The proof is completed here.
\end{proof}
\begin{rem}\label{rem:degeneracy}
	For any random variable $X\sim\pi_n$ and Boolean functions $\f_n:\Omega_n\rightarrow\{0,1\}$, following similar computations as in \eqref{eq:CTMC_E-V}, one can show
	\[\var(\f_n):=\var(\f_n(X))=\displaystyle\sum_{j=1}^{|\Omega_n|-1}|\langle \f_n,\psi_{n,j}\rangle|^2.\]
	\begin{itemize}
		\item $0\leq \cov\left(\f_n(X_n(0)),\f_n(X_n(t)\right)\leq \displaystyle\sum_{j=1}^{|\Omega_n|-1}|\langle \f_n,\psi_{n,j}\rangle|^2=\var(\f_n)$, \;by Lemma \ref{lem:cov-ns}.
		\item $0\leq \mathbb{P}\left(\f_n(X_n(0))\neq \f_n(X_n(t))\right)\leq 2\displaystyle\sum_{j=1}^{|\Omega_n|-1}|\langle \f_n,\psi_{n,j}\rangle|^2=2\var(\f_n)$, \;by Lemma \ref{lem:boolean-st}.
%		Thus, the sequence $\{\f_n\}_n$ is noise sensitive as well as noise stable with respect to $(X_n(t),\Omega_n,\pi_n)$ if $\var(\f_n)\rightarrow 0$. Such sequence is known as \emph{degenerate}.
	\end{itemize}
	Thus, if $\var(\f_n)\rightarrow 0$ then the sequence is both noise sensitive and noise stable with respect to $(X_n(t),\Omega_n,\pi_n)$. Such sequences are usually called \emph{degenerate}.
\end{rem}
We are now in a position to recall two results from \cite{Forsstrom_arXiv}, which will be used as alternative definitions for noise sensitivity and stability. The proofs are straightforward applications of Lemma \ref{lem:cov-ns} and Lemma \ref{lem:boolean-st}.
\begin{lem}{\cite[Proposition 1.4]{Forsstrom_arXiv}}\label{lem:spectral_NS}
	Consider the continuous-time Markov chain $(X_n(t),\Omega_n,\pi_n)$ from above, and Boolean functions $\f_n:\Omega_n\rightarrow\{0, 1\},\;n\geq 1$. The sequence $\{\f_n\}_n$ is noise sensitive with respect to $\{(X_n(t),\Omega_n,\pi_n)\}_n$ if and only if, for every integer $k > 0$,
	\[\lim_{n\rightarrow\infty}\sum_{j:\;\theta_{n,1}\leq\theta_{n,j}<k\theta_{n,1}}|\langle \f_n,\psi_{n,j}\rangle|^2=0.\]
\end{lem}
\begin{lem}{\cite[Proposition 1.6]{Forsstrom_arXiv}}\label{lem:spectral_NStable}
	A sequence $\{\f_n\}_n$ of Boolean functions, $\f_n:\Omega_n\rightarrow\{0, 1\}$, is noise stable with respect to $\{(X_n(t),\Omega_n,\pi_n)\}_n$ if and only if for all $\delta>0$ there is a positive integer $k$  such that
	\[\sup_n\sum_{j:\;\theta_{n,j}\geq k\theta_{n,1}}|\langle \f_n,\psi_{n,j}\rangle|^2<\delta.\]
\end{lem}
Lemma \ref{lem:spectral_NS} is analogous to the classical result on the hypercube case, stating that a function is noise sensitive if and only if most of its Fourier weight is on the ``high energy" eigenfunctions of the random walk operator. As a degenerate sequence is both noise sensitive and noise stable, \emph{we will only be interested in non-degenerate sequence of Boolean functions.}
%%%%%%%
%%%%%%
%%%%%%%

\subsection{Representation theory background}\label{subsec:repn._theory_bkground}
In this section, we recall the representation theory of a finite group. Let $V$ be a finite-dimensional complex vector space and GL$(V)$ be the group of all invertible linear operators from $V$ to itself under composition of linear mappings.  Unless otherwise stated, all the vector spaces considered in this paper are finite-dimensional. Elements of GL$(V)$ can be thought of as invertible matrices over $\mathbb{C}$. Let $G$ be a finite group. Let $I$ denote the identity element of GL$(V)$ (i.e. the identity operator on $V$) and $\Gid$ denote the identity element of $G$. A (complex) \emph{linear representation} $(\rho,V)$ of $G$ is a homomorphism $\rho:G\rightarrow \text{GL}(V)$, i.e., $\rho(g_1g_2)=\rho(g_1)\rho(g_2)\text{ for all }g_1,\;g_2\in G$. The representation space $V$ is called the \emph{$G$-module} corresponding to the representation $\rho$. Given $\rho$, we simply say $V$ is a representation of $G$. Two useful examples are given below:
\begin{itemize}
	\item Let $V$ be one-dimensional. Then, $\trivial:G\rightarrow \text{GL}(V)$ defined by $\trivial(g)\mapsto(v\mapsto v)$ for all $v\in V$ and $g\in G$, is a representation of $G$, known as the \emph{trivial representation} of $G$. %In fact, $V$ one-dimensional implies that GL$(V)$ is isomorphic to $\mathbb{C}^*$ (the group of non-zero complex numbers under usual multiplication). Therefore t
	Alternatively, the trivial representation $\trivial:G\rightarrow\mathbb{C}^*$ can also be defined by $\trivial(g)=1$ for all $g\in G$.
	\item Let $\mathbb{C}^G$ denote the complex vector space of all complex valued functions defined on $G$, i.e., $\mathbb{C}^G=\{\f:G\rightarrow \mathbb{C}\}$. Then, the \emph{right regular representation} $R:G\longrightarrow \text{GL}(\mathbb{C}^G)$ of $G$ is defined by $g\mapsto R_g$\footnote{From now on, for the right regular representation, we write $R_g$ to denote $R(g)$.} for $g\in G$, where  $R_g:\mathbb{C}^G\rightarrow \mathbb{C}^G$ is defined by
	\[(R_g\f)(x)=\f(xg)\text{ for all }x\in G\text{ and }\f\in \mathbb{C}^G.\]
	Setting a basis $\{\mathbb{1}_{\{g\}}:g\in G\}$,  we can think of $R_g$ as an invertible matrix over $\mathbb{C}$ of order $|G|\times|G|$. The \emph{left regular representation} can be defined in a similar fashion.
\end{itemize}
The dimension of the vector space $V$, denoted $d_{\rho}$, is called the \emph{dimension} of the representation $\rho$. The trace of the matrix $\rho(g)$ is said to be the \emph{character} value of $\rho$ at $g$, it is denoted by $\chi^{\rho}(g)$. %The character values are constants on conjugacy classes, i.e., the characters are class functions.
Recall that $\chi^{\rho}(\Gid)=d_{\rho}$ and $\chi^{\rho}(g^{-1})=\overline{\chi^{\rho}(g)}$, the complex conjugate of $\chi^{\rho}(g)$. A vector subspace $U$ of $V$ is said to be \emph{stable} (or \emph{invariant}) under $\rho$ if $\rho(g)\left(U\right)\subset U$ for all $g$ in $G$. \emph{Given a vector subspace $U$ of $V$ stable under $\rho$, there exists a complement $U^0$ of $U$ in $V$ which is stable under $\rho$} (\cite[Theorem 1]{Serre}). The representation $\rho$ is \emph{irreducible} if $V$ is non-trivial, and it does not has a non-trivial proper subspace which is stable under $\rho$; an example is the trivial representation defined above. Two representations $(\rho_1,V_1)$ and $(\rho_2,V_2)$ of $G$ are said to be \emph{isomorphic} if there exists an invertible linear map $T:V_1\rightarrow V_2$ such that $T\circ\rho_1(g)=\rho_2(g)\circ T$ for all $g\in G$. %i.e., the following diagram commutes for all $g\in G$:
%\[\begin{tikzcd}
	%	V_1\arrow{r}{\rho_1(g)}\arrow[swap]{d}{T} & V_1\arrow{d}{T}\\
	%	V_2\arrow{r}{\rho_2(g)} & V_2
	%\end{tikzcd}\]
\emph{Two representations of a finite group are isomorphic if and only if they have the same characters}. Let $H$ be a subgroup of $G$. The \emph{restriction} of the representation $\rho$ to $H$ is denoted by $\rho\downarrow^G_H$ and is defined by $\rho\downarrow^G_H(h):=\rho(h)$ for all $h\in H$. If $\chi^{\rho}$ is the character of $\rho$, then the character of the restriction 
$\rho\downarrow^G_H$ is denoted by $\chi^\rho\downarrow^G_H$.
The \emph{direct sum of the representations} $(\rho_1,V_1)$ and $(\rho_2,V_2)$, denoted $\rho_1\oplus\rho_2$, is the representation $(\rho_1\oplus\rho_2):G\rightarrow \text{GL}(V_1\oplus V_2)$ defined by, %$(\rho_1\oplus\rho_2)(g)(v_1\oplus v_2)=\rho_1(g)(v_1)\oplus\rho_2(g)(v_2)$ for $v_1\in V_1,v_2\in V_2$ and $g\in G$.
\[(\rho_1\oplus\rho_2)(g)(v_1\oplus v_2)=\rho_1(g)(v_1)\oplus\rho_2(g)(v_2) \text{ for }v_1\in V_1,v_2\in V_2\text{ and }g\in G.\]
The character of $\rho_1\oplus\rho_2$, denoted $\chi^{\rho_1\oplus\rho_2}$, given by $\chi^{\rho_1\oplus\rho_2}=\chi^{\rho_1}+\chi^{\rho_2}$. 
%\begin{lem}[{Schur's lemma \cite[Proposition 5]{Serre}}]\label{lem:ch2_Schurs_lemma}
	%	Let $(\rho_1,V_1)$ and $(\rho_2,V_2)$ be two \emph{(}complex\emph{)} irreducible representations of the finite group $G$. If a linear mapping $\Phi:V_1\rightarrow V_2$ satisfies $\rho_2(g)\circ\Phi=\Phi\circ\rho_1(g)$ for all $g\in G$, then
	%	\begin{enumerate}
		%		\item $\Phi=0$, when $\rho_1$ and $\rho_2$ are non-isomorphic.
		%		\item  $\Phi$ is a scalar multiple of the identity, when $V_1=V_2$ and $\rho_1=\rho_2$.
		%	\end{enumerate}
	%\end{lem}
One key result in this this context ensures that \emph{every $($complex$)$ linear representation is a direct sum of irreducible representations} (\cite[Theorem 1]{Serre}). Moreover this decomposition is unique up to isomorphism of representation. Another important result is the Schur's lemma, which says that \emph{Given any irreducible representation $(\rho^{\emph{irr}}, V^{\emph{irr}})$ of $G$, if a linear operator $T:V^{\emph{irr}}\rightarrow V^{\emph{irr}}$ satisfies $T\circ \rho^{\emph{irr}}(g)=\rho^{\emph{irr}}(g)\circ T$ for all $g\in G$, then $T$ is a scalar multiple of identity operator on $V^{\emph{irr}}$} \cite[Proposition 5]{Serre}. We now define an inner product on $\mathbb{C}^G$, the vector space of complex valued functions defined on $G$. The inner product $\langle\cdot\mid\cdot\rangle$ is defined by
\begin{equation}\label{eq:ChIP}
	\langle\f\mid\g\rangle:=\frac{1}{|G|}\displaystyle\sum_{x\in G}\f(x)\overline{\g(x)},\quad\f,\;\g\in \mathbb{C}^G.
\end{equation}
\begin{thm}[{\cite[Theorem 6]{Serre}}]\label{thm:ch2_irreducible_character_as_ONB}
	The characters corresponding to the non-isomorphic irreducible representations of $G$ form an $\langle\cdot\mid\cdot\rangle$-orthonormal basis of the complex vector space of the class functions of $G$.
\end{thm}
Let $G$ be a finite group and $\widehat{G}$ be the set of all non-isomorphic irreducible representations of $G$. Theorem \ref{thm:ch2_irreducible_character_as_ONB} ensures that, the number of irreducible representations of a finite group is equal to the number of its conjugacy classes. If $V\cong m_1W_1\oplus\cdots\oplus m_{\ell}W_{\ell}$ is the decomposition of the representation $(\rho,V)$ into irreducible representations of $\rho$, then $\langle\chi^{\rho}\mid\chi^{\rho}\rangle=m_1^2+\cdots+m_{\ell}^2$. Moreover, if $\chi^i$ denotes the irreducible character of $W_i,\;1\leq i\leq\ell$, then $m_i=\langle\chi\mid\chi^i\rangle$ is called the multiplicity of $W_i$ in the decomposition of $V$. \emph{The regular $($true for both left and right$)$ representation of $G$ decomposes into irreducible representations with multiplicity equal to their respective dimensions} \cite[p. 18, Corollary 1]{Serre}. Thus we have,
\begin{equation}\label{eq:Group_alg._decom.}
	\mathbb{C}^G\cong\underset{\rho\in\widehat{G}}{\oplus}\;d_{\rho}W^{\rho},
\end{equation}
where $W^{\rho}$ is the irreducible $G$-module corresponding to $\rho\in\widehat{G}$, and $d_{\rho}=\dim(W^{\rho})$. Also, equating the dimension in \eqref{eq:Group_alg._decom.}, we obtain $|G|=\displaystyle\sum_{\rho\in \widehat{G}}d_{\rho}^2$.
%\begin{thm}[{Frobenius reciprocity \cite[Theorem 1.12.6]{Sagan}}]\label{thm:ch2_Frobenious_Reciprocity}
	%	Let $H$ be a subgroup of the finite group $G$ and suppose that $\psi$ and $\chi$ are characters of $H$ and $G$ respectively. Then
	%	\[\langle\psi\uparrow_H^G,\chi\rangle=\langle\psi,\chi\downarrow_H^G\rangle,\]
	%	where the left inner product is calculated in $G$ and the right one in $H$.
	%\end{thm}
\subsection{Fourier analysis and continuous-time random walks on finite groups}\label{sec:non-comm._Fou._an.}
We will write down an orthonormal basis of $\mathbb{C}^G$, the set of all complex valued functions defined on $G$. First we recall some definition and standard results on the Fourier analysis of $G$. Let $\f,\g:G\rightarrow\mathbb{C}$ be two functions on the finite group $G$. The \emph{convolution} of $\f$ and $\g$, denoted  $\f*\g$, is defined by \[(\f*\g)(x):=\displaystyle\sum_{y\in G}\f(xy^{-1})\g(y).\]
Let $(\rho,V)$ be  a (complex) linear representation of $G$. Then the \emph{Fourier transform} of $\f$ at $\rho$, denoted $\widehat{\f}$, is defined by 
\[\widehat{\f}(\rho):=\displaystyle\sum_{x\in G}\f(x)\rho(x).\]
For example, recall the right regular representation $R$ of $G$. Given any function $\h:G\rightarrow\mathbb{C}$, the operator $\widehat{\h}(R)$ is defined by
	\[(\widehat{\h}(R)\f)(x)=\sum_{u\in G}\f(xu)\h(u)\text{ for all }x\in G\text{ and }\f\in \mathbb{C}^G.\]
It can be easily seen that $\widehat{(\f*\g)}(\rho)=\widehat{\f}(\rho)\widehat{\g}(\rho)$ for all $\f,\g\in\mathbb{C}^G$. The Plancherel formulae (cf. \cite[Theorem 4.1]{D1}) are given by
\begin{align}
	\f(x)&=\frac{1}{|G|}\sum_{\rho\in\widehat{G}}d_{\rho}\Tr\left(\rho(x^{-1})\widehat{\f}(\rho)\right),\label{eq:Inv_Fourier_transf}\\
	\sum_{x\in G}\f(x^{-1})\g(x)&=\frac{1}{|G|}\sum_{\rho\in\widehat{G}}d_{\rho}\Tr\left(\widehat{\f}(\rho)\widehat{\g}(\rho)\right).\label{eq:Plancherel formula}
\end{align}
To write down an $\langle\;|\;\rangle$ -- orthonormal basis of $\mathbb{C}^G$, we fix some notation first. Throughout this article, we use the notation $\delta_{*,*}$ to denote the \emph{Kronecker} symbol. Recall that $\widehat{G}$ is the set of all non-isomorphic irreducible representations of $G$. For every $\rho\in\widehat{G}, \; W^{\rho}$ denotes the associated irreducible $G$-module, and $d_{\rho}=\dim(W^{\rho})$. Let $\mathscr{S}_{\rho}$ be a fixed set of cardinality $d_{\rho}$. Since $\dim(W^{\rho}) = d_{\rho}$, every basis of $W^{\rho}$ can be indexed by the elements of $\mathscr{S}_{\rho}$. Throughout this paper, we shall use $\mathscr{S}_{\rho}$ as the indexing set for basis vectors of $W^{\rho}$. Let us consider a basis $\mathcal{B}^{\rho}:=\{v_{\alpha}:\alpha\in\mathscr{S}_{\rho}\}$ of $W^{\rho}$, such that the representation matrices $\big[\rho(g)\big]_{\mathcal{B}^{\rho}},\;g\in G$, are unitary with respect to $\mathcal{B}^{\rho}$; to ensure existence of such basis see \cite[p.6]{Serre}. \emph{In this paper, we always work with bases with respect to which all the representation matrices are unitary}. Let us define $d_{\rho}^2$ linear operators $E_{\mathcal{B}{\rho}}^{\alpha\beta}:W^{\rho}\rightarrow W^{\rho},\;\alpha,\beta\in\mathscr{S}_{\rho}$. Explicitly, for $\alpha,\beta\in\mathscr{S}_{\rho},\; E_{\mathcal{B}{\rho}}^{\alpha\beta}:W^{\rho}\rightarrow W^{\rho}$ is defined on the basis elements by
\[E^{\alpha\beta}_{\mathcal{B}_{\rho}}(v_{\gamma})=\delta_{\beta,\gamma}v_{\alpha}\quad\text{ for }\quad\gamma\in\mathscr{S}_{\rho},\]
 i.e., the matrix, denoted $\big[E^{\alpha\beta}_{\mathcal{B}_{\rho}}\big]_{\mathcal{B}_{\rho}}$, of the operator $E^{\alpha\beta}_{\mathcal{B}_{\rho}}$ with respect to basis $\mathcal{B}_{\rho}$ is the $d_{\rho}\times d_{\rho}$ matrix having only non-zero entry $1$ at row $\alpha$ and column $\beta$.
 
 For every $\rho\in\widehat{G}$, we define $d_{\rho}^2$ functions $\psi^{\alpha\beta}_{\mathcal{B}_{\rho}},\;\alpha,\beta\in\mathcal{S}_{\rho}$, through their Fourier transforms at the irreducible representations of $G$ by 
 \begin{equation}\label{eq:FT_of_ONB}
 	\widehat{\psi}^{\alpha\beta}_{\mathcal{B}_{\rho}}(\lambda):=\delta_{\rho,\lambda}\frac{|G|}{\sqrt{d_{\rho}}}E^{\alpha\beta}_{\mathcal{B}_{\rho}}\;\text{ for all }\lambda\in\widehat{G},
 \end{equation}
 thanks to \eqref{eq:Inv_Fourier_transf}. More explicitly, for $\alpha,\beta\in\mathscr{S}_{\rho}$,
 \begin{equation}\label{eq:ONFunctions}
 	\psi^{\alpha\beta}_{\mathcal{B}_{\rho}}(x):=\frac{1}{|G|}\sum_{\lambda\in\widehat{G}}d_{\lambda}\Tr\left(\lambda(x^{-1})\delta_{\rho,\lambda}\frac{|G|}{\sqrt{d_{\rho}}}E^{\alpha\beta}_{\mathcal{B}_{\rho}}\right)=\sqrt{d_{\rho}}\Tr\left(\rho(x^{-1})E^{\alpha\beta}_{\mathcal{B}_{\rho}}\right).
 \end{equation}
 \begin{rem}
 	The functions listed in \eqref{eq:ONFunctions} are real valued if, for every $x\in G$, the matrix of the operator $\rho(x)$ with respect to $\mathcal{B}_{\rho}$ has real entries.
 \end{rem}
 \begin{lem}\label{lem:ONB_for_G}
 	Recalling the notations from the precdding paragraphs, the set
 	\begin{equation}\label{eq:ONB_for_G}
 		\underset{\rho\in\widehat{G}}{\bigcup}\big\{\psi^{\alpha\beta}_{\mathcal{B}_{\rho}}:\alpha,\beta\in\mathscr{S}_\rho\big\}
 	\end{equation}
 	forms an $\langle\;|\;\rangle$ -- orthonormal basis of $\mathbb{C}^G$.
 \end{lem}
 \begin{proof}
 	The set given in \eqref{eq:ONB_for_G} contains $\sum_{\rho\in\widehat{G}}d_{\rho}^2=|G|$ vectors. As $|G|$ is the dimension of the vector space $\mathbb{C}^G$, it is enough to show that \eqref{eq:ONB_for_G} is an orthonormal set of vectors in $\mathbb{C}^G$. 
 	
 	For $\rho\in\widehat{G}$, recall the bases $\mathcal{B}_{\rho}$ of $W^{\rho}$. Given $\f\in\mathbb{C}^G$, if we define $\check{\f}:G\rightarrow \mathbb{C}$ by $\check{\f}(x)=\overline{\f(x^{-1})}$, then the matrix of the Fourier transform of $\check{\f}$ at the irreducible representation $\rho$ with respect to the basis $\mathcal{B}_{\rho}$ is given by
 	\begin{equation}\label{eq:ONB_for_G1}
 		\Big[\widehat{\check{\f}}(\rho)\Big]_{\mathcal{B}_{\rho}}:=\Big[\widehat{\f}(\rho)\Big]_{\mathcal{B}_{\rho}}^{\dagger}\text{, the conjugate-transpose  of the matrix }\Big[\widehat{\f}(\rho)\Big]_{\mathcal{B}_{\rho}}.
 	\end{equation}
 	This follows from the fact that $\big[\rho(g)\big]_{\mathcal{B}^{\rho}}$ is unitary matrix for every $g\in G$.
 	
 	Let $\rho,\rho'\in\widehat{G}$, recall that $\mathscr{S}_{\rho}$ (respectively, $\mathscr{S}_{\rho'}$) denotes the indexing set of the basis $\mathcal{B}^{\rho}$ (respectively, $\mathcal{B}^{\rho'}$) of $W^{\rho}$ (respectively, $W^{\rho'}$). Then, for $\alpha,\beta\in\mathscr{S}_{\rho}$ and $\alpha',\beta'\in\mathscr{S}_{\rho}$, we compute
 	\begin{align*}
 		\langle\psi^{\alpha\beta}_{\mathcal{B}_{\rho}}\mid \psi^{\alpha'\beta'}_{\mathcal{B}_{\rho'}}\rangle&=\frac{1}{|G|}\sum_{x\in G}\psi^{\alpha\beta}_{\mathcal{B}_{\rho}}(x)\overline{\psi^{\alpha'\beta'}_{\mathcal{B}_{\rho'}}(x)}\\
 		&=\frac{1}{|G|}\sum_{x\in G}\psi^{\alpha\beta}_{\mathcal{B}_{\rho}}(x)\check{\psi}^{\alpha'\beta'}_{\mathcal{B}_{\rho'}}(x^{-1}),\quad\text{ using }\check{\psi}^{\alpha'\beta'}_{\mathcal{B}_{\rho'}}(x):=\overline{\psi^{\alpha'\beta'}_{\mathcal{B}_{\rho'}}(x^{-1})}\\
 		&=\frac{1}{|G|^2}\sum_{\lambda\in \widehat{G}}d_{\lambda}\Tr\left(\widehat{\psi}^{\alpha\beta}_{\mathcal{B}_{\rho}}(\lambda)\widehat{\check{\psi}}^{\alpha'\beta'}_{\mathcal{B}_{\rho'}}(\lambda)\right),\text{ using }\eqref{eq:Plancherel formula}
 	\end{align*}
 	We have seen that \eqref{eq:ONB_for_G1} implies $\Big[\widehat{\check{\psi}}^{\alpha'\beta'}_{\mathcal{B}_{\rho'}}(\lambda)\Big]_{\mathcal{B}_{\rho'}}=\Big[\widehat{\psi}^{\alpha'\beta'}_{\mathcal{B}_{\rho'}}(\lambda)\Big]_{\mathcal{B}_{\rho'}}^{\dagger}$. Therefore, using \eqref{eq:FT_of_ONB}, we have
 	\begin{align*}
 		\langle\psi^{\alpha\beta}_{\mathcal{B}_{\rho}}\mid \psi^{\alpha'\beta'}_{\mathcal{B}_{\rho'}}\rangle&=\frac{1}{|G|^2}\sum_{\lambda\in \widehat{G}}d_{\lambda}\Tr\left(\delta_{\rho,\lambda}\frac{|G|}{\sqrt{d_{\rho}}}\big[E^{\alpha\beta}_{\mathcal{B}_{\rho}}\big]_{\mathcal{B}_{\rho}}\cdot \delta_{\rho',\lambda}\frac{|G|}{\sqrt{d_{\rho'}}}\big[E^{\beta'\alpha'}_{\mathcal{B}_{\rho'}}\big]_{\mathcal{B}_{\rho'}}\right)\\
 		&=\delta_{\rho,\rho'}\Tr\left(\big[E^{\alpha\beta}_{\mathcal{B}_{\rho}}\big]_{\mathcal{B}_{\rho}}\cdot \big[E^{\beta'\alpha'}_{\mathcal{B}_{\rho'}}\big]_{\mathcal{B}_{\rho'}}\right)=\delta_{\rho,\rho'}\delta_{\alpha,\alpha'}\delta_{\beta,\beta'}.
 	\end{align*}
 	This completes the proof.
 \end{proof}
A \emph{continuous-time random walk on a finite group $G$ generated by a  rate function $\mu$} is the continuous-time Markov chain with state space $G$, such that the transition from $x$ to $y$ happens with rate $\mu(x^{-1}y)$, $x,y\in G$.  Therefore,  the infinitesimal generator is the matrix
\[Q:=\left(\mu(x^{-1}y)-\delta_{x,y}\sum_{g\in G}\mu(g)\right)_{x,y\in G},\]
where $\delta_{*,*}$ is the \emph{Kronecker delta}. It can be easily seen that $-Q$ is the matrix of the operator 
\[\left(\displaystyle\sum_{g\in G}\mu(g)\right)I-\widehat{\mu}(R)=\displaystyle\sum_{g\in G}\mu(g)\left(I-R_g\right)\]
on $\mathbb{C}^G$ with respect to the basis $\{\mathbb{1}_{\{g\}}:g\in G\}$. The random walk is irreducible \emph{if and only if the support of $\mu$ generates $G$} \cite[Proposition 2.3]{S}. Also, the uniform distribution is the \emph{stationary} distribution of the random walk. Thus, the inner products \eqref{eq:MCIP} and \eqref{eq:ChIP} coincide for continuous-time random walks on a finite groups. From now on, we denote this inner product using notation $\langle\cdot,\cdot\rangle$. 
We now prove a lemma that we use later in this paper.
\begin{lem}\label{lem:FT_of_const._1_fn}
	Recall that $\widehat{G}$ is the set of equivalence classes of irreducible representations of the finite group $G$. Let $\mathbf{1}:G\rightarrow\{0,1\}$ be defined by $\mathbf{1}(g)=1$ for all $g\in G$, and $\delta_{*,*}$ denotes the Kronecker delta. Then
	\[\widehat{\mathbf{1}}(\sigma)=\frac{|G|}{d_{\sigma}}\;\delta_{\sigma,\trivial}\;I_{\sigma},\quad\sigma\in\widehat{G}.\]
	Here, $I_{\sigma}$ denotes the identity operator on the irreducible $G$-module indexed by $\sigma$, and $d_{\sigma}$ denotes its dimension.
\end{lem}
\begin{proof}
	Let $(\rho_{\sigma},V^{\sigma})$ denote the irreducible representation indexed by $\sigma$. Then, 
	\[\widehat{\mathbf{1}}(\sigma):=\widehat{\mathbf{1}}(\rho_{\sigma})=\sum_{h\in G}\rho_{\sigma}(h)\quad\text{ implies that }\quad\widehat{\mathbf{1}}(\sigma)\circ\rho_{\sigma}(g)=\rho_{\sigma}(g)\circ\widehat{\mathbf{1}}(\sigma)\text{ for all }g\in G.\]
	Thus, by Schur's lemma, we have $\widehat{\mathbf{1}}(\rho_{\sigma})=c I_{\sigma}$ for some $c\in\mathbb{C}$. The value of $c$ can be obtained by equating the traces of $\widehat{\mathbf{1}}(\rho_{\sigma})g$ and $c I_{\sigma}$, i.e., $\displaystyle\sum_{h\in G}\chi^{\sigma}(h)=c d_{\sigma}$, and Theorem \ref{thm:ch2_irreducible_character_as_ONB}.
\end{proof}
\section{Examples of some noise sensitive or noise stable Boolean functions}\label{sec:examples}
We devote this section to study some examples. We first focus on the complete graph with $n$ vertices, denoted $K_n$. We start with examples of Boolean functions on $S_n$, where the noise is generated by the interchange process on the complete graph $K_n$. We use the notation $\{X_{\mathbb{K}}(t)\}_{t\geq 0}$ to denote the interchange process on $K_n$. In our case, all the Poisson clocks ring at the same constant rate $\nu$, and thus the relaxation time is $t^{\mathbb{K}}_{\text{rel}}=\frac{1}{\nu n}$ (set $\lambda=(n-1,1)$ in Lemma \ref{lem:compgrp_eigval_on_irr}; cf. the Aldous-Caputo-Liggett-Richthammer theorem \cite{CLR}). Moreover, the uniform distribution on $S_n$, denoted by $U_n$, is the stationary distribution. Since $X_{\mathbb{K}}(0)\sim U_n$, it follows that $X_{\mathbb{K}}\left(\alpha t^{\mathbb{K}}_{\text{rel}}\right)\sim U_n$. The examples are given below.
\begin{exmp}\label{ex:parity}
	Let $A_n$ be the alternating group, i.e., the subgroup of $S_n$ consisting of all even permutations. The parity function $\phi_n\longrightarrow \{0,1\}$, given by $\phi_n:=\mathbb{1}_{A_n}$, is noise sensitive with respect to the interchange process on $K_n$.
\end{exmp}
\begin{proof}
	Let $\alpha>0$ be arbitrary. We first observe that
	\begin{equation}\label{eq:parity-1}
		\mathbb{E}\left(\phi_n\left(X_{\mathbb{K}}(0)\right)\right)=\frac{1}{2}\quad\text{and}\quad\mathbb{E}\left(\phi_n\left(X_{\mathbb{K}}\left(\alpha t^{\mathbb{K}}_{\text{rel}}\right)\right)\right)=\frac{1}{2}.
	\end{equation}
	The function $\phi_n$ changes its value with every transition, and transitions occur when the Poisson clocks at the edges ring. The number of transitions happen according to a Poisson process with rate $\nu \binom{n}{2}$ (as the minimum of $\binom{n}{2}$ independent exponential$(\nu)$ random variables is an exponential random variable with parameter $\nu\binom{n}{2})$. Thus, the number of transitions by time $\alpha t^{\mathbb{K}}_{\text{rel}}=\frac{\alpha}{\nu n}$, denoted $N\left(\alpha t^{\mathbb{K}}_{\text{rel}}\right)$, is a Poisson random variable with parameter $\nu\binom{n}{2} \frac{\alpha}{\nu n}=(n-1)\frac{\alpha}{2}$. Therefore,
	\begin{align}\label{eq:parity-2.1}
		&\mathbb{E}\left(\phi_n\left(X_{\mathbb{K}}(0)\right)\phi_n\left(X_{\mathbb{K}}\left(\alpha t^{\mathbb{K}}_{\text{rel}}\right)\right)\right)\nonumber\\
		=&\sum_{\pi,\sigma\in S_n}\phi_n(\pi)\phi_n(\sigma)\mathbb{P}\left(X_{\mathbb{K}}\left(\alpha t^{\mathbb{K}}_{\text{rel}}\right)=\sigma|X_{\mathbb{K}}(0)=\pi\right)\mathbb{P}\left(X_{\mathbb{K}}(0)=\pi\right)\nonumber\\
		=&\sum_{\pi,\sigma\in A_n}1\times 1\times \mathbb{P}\left(X_{\mathbb{K}}\left(\alpha t^{\mathbb{K}}_{\text{rel}}\right)=\sigma|X_{\mathbb{K}}(0)=\pi\right)\times\frac{1}{n!}\nonumber\\
		=&\frac{1}{n!}\sum_{\pi\in A_n}\sum_{\sigma\in A_n}\mathbb{P}\left(X_{\mathbb{K}}\left(\alpha t^{\mathbb{K}}_{\text{rel}}\right)=\sigma|X_{\mathbb{K}}(0)=\pi\right)=\frac{1}{n!}\sum_{\pi\in A_n}\mathbb{P}\left(X_{\mathbb{K}}\left(\alpha t^{\mathbb{K}}_{\text{rel}}\right)\in A_n|X_{\mathbb{K}}(0)=\pi\right).
	\end{align}
	But, we know that
	\begin{align}\label{eq:parity-2.2}
		\mathbb{P}\left(X_{\mathbb{K}}\left(\alpha t^{\mathbb{K}}_{\text{rel}}\right)\in A_n|X_{\mathbb{K}}(0)=\pi\right)=&\mathbb{P}\left(N\left(\alpha t^{\mathbb{K}}_{\text{rel}}\right)\text{ is even}\right)\nonumber\\
		=&\sum_{k=0}^{\infty}e^{-(n-1)\frac{\alpha}{2}}\cdot\frac{\left((n-1)\frac{\alpha}{2}\right)^{2k}}{(2k)!}\nonumber\\
		=&e^{-(n-1)\frac{\alpha}{2}}\cdot\frac{e^{(n-1)\frac{\alpha}{2}}+e^{-(n-1)\frac{\alpha}{2}}}{2}=\frac{1}{2}\left(1+e^{-(n-1)\alpha}\right)
	\end{align}
	Therefore, using \eqref{eq:parity-2.2} in \eqref{eq:parity-2.1} we obtain
	\begin{equation}\label{eq:parity-2}
		\mathbb{E}\left(\phi_n\left(X_{\mathbb{K}}(0)\right)\phi_n\left(X_{\mathbb{K}}\left(\alpha t^{\mathbb{K}}_{\text{rel}}\right)\right)\right)=\frac{1}{n!}\sum_{\pi\in A_n}\frac{1}{2}\left(1+e^{-(n-1)\alpha}\right)=\frac{1}{4}\left(1+e^{-(n-1)\alpha}\right)
	\end{equation}
	Finally from \eqref{eq:parity-1} and \eqref{eq:parity-2}, we have
	\begin{align*}
		&\cov\left(\phi_n\left(X_{\mathbb{K}}(0)\right),\;\phi_n\left(X_{\mathbb{K}}\left(\alpha t^{\mathbb{K}}_{\text{rel}}\right)\right)\right)\\
		=&\mathbb{E}\left(\phi_n\left(X_{\mathbb{K}}(0)\right)\phi_n\left(X_{\mathbb{K}}\left(\alpha t^{\mathbb{K}}_{\text{rel}}\right)\right)\right)-\mathbb{E}\left(\phi_n\left(X_{\mathbb{K}}(0)\right)\right)\mathbb{E}\left(\phi_n\left(X_{\mathbb{K}}\left(\alpha t^{\mathbb{K}}_{\text{rel}}\right)\right)\right)\\
		=&\frac{1}{4}\left(1+e^{-(n-1)\alpha}\right)-\frac{1}{2}\cdot\frac{1}{2}=\frac{1}{4}e^{-(n-1)\alpha}\longrightarrow 0\;\text{ as }n\rightarrow\infty.\qedhere
	\end{align*}
\end{proof}
\begin{exmp}\label{ex:int_parity}
	The function $\phi_n:S_n\longrightarrow \{0,1\}$ defined by $\phi_n:=\mathbb{1}_{\{\pi\in S_n:\pi(n)\text{ is even}\}}$ is noise stable with respect to the interchange process on $K_n$.
\end{exmp}
\begin{proof}
	We first define two subsets of $[n]:=\{1,2,\dots,n\}$ as follows
	\[E_n:=\{i\in[n]:i\text{ is even}\}\quad\text{and}\quad O_n:=\{i\in[n]:i\text{ is odd}\}.\]
	Therefore, writing $\epsilon_n=\frac{1-(-1)^n}{2n}$, we obtain
	\begin{equation}\label{eq:int_parity-1}
		\frac{|E_n|}{n}=\frac{2n-1+(-1)^n}{4n}=\frac{1-\epsilon_n}{2}\quad\text{and}\quad\frac{|O_n|}{n}=\frac{2n+1-(-1)^n}{4n}=\frac{1+\epsilon_n}{2}.
	\end{equation}
	For $1\leq i<j\leq n$, let $T_{i,j}$ denote the exponential random variable representing the time when Poisson clock at edge $\{i,j\}$ rings. If $\sigma\in S_n$ with $\sigma(n)$ even, then the value of $\phi_n$ changes from $1$ to $0$ whenever the Poisson clock at the edge $\{n,\sigma^{-1}(a)\}$ rings for some $a\in O_n$, this happens at time 
	\[\min\{T_{n,\sigma^{-1}(a)}:a\in O_n\}\sim\text{exponential}\left(\nu|O_n|\right).\]
	Again, for $\sigma\in S_n$ with $\sigma(n)$ odd, the value of $\phi_n$ changes from $0$ to $1$ whenever the Poisson clock at the edge $\{n,\sigma^{-1}(a)\}$ rings for some $a\in E_n$, this happens at time 
	\[\min\{T_{n,\sigma^{-1}(a)}:a\in E_n\}\sim\text{exponential}\left(\nu|E_n|\right).\]
	Thus, the process $Y(t):=\phi_n(X_{\mathbb K}(t))$ is itself a two-state continuous-time Markov chain, with infinitesimal generator
\[
Q_n
=
\begin{array}{c@{\quad}c}
	& \begin{array}{cc}
		0 &\;\qquad\; 1
	\end{array}
	\\[1mm]
	\begin{array}{c}
		0\\
		1
	\end{array}
	&
	\left(
	\begin{array}{cc}
		-\nu| E_n| & \nu| E_n|\\
		\nu| O_n| & -\nu| O_n|
	\end{array}
	\right)
\end{array}.
\]
Here, the states are ordered as $0,1$. It follows that
\begin{equation}
e^{tQ_n}
=
\begin{array}{c@{\quad}c}
	& \begin{array}{cc}
		0 &\qquad\qquad\qquad\qquad 1
	\end{array}
	\\[1mm]
	\begin{array}{c}
		0\\[0.65cm]
		1
	\end{array}
	&
	\left(
	\begin{array}{cc}
		\dfrac{| O_n|}{n}
		+\dfrac{| E_n|}{n} e^{-\nu nt}
		
		&
		\dfrac{| E_n|}{n}
		\left(1-e^{-\nu nt}\right)
		\\[3mm]
		\dfrac{| O_n|}{n}
		\left(1-e^{-\nu nt}\right)
		&
		\dfrac{| E_n|}{n}
		+\dfrac{| O_n|}{n} e^{-\nu nt}
	\end{array}
	\right)
\end{array},\qquad\text{ for all }t>0.
\end{equation}
Therefore, for $\alpha>0$, we have
\begin{align}
	&\mathbb{P}\left(\phi_n\left(X_{\mathbb{K}}(0)\right)\neq \phi_n\left(X_{\mathbb{K}}\left(\alpha t^{\mathbb{K}}_{\text{rel}}\right)\right)\right)\nonumber\\
	=&\mathbb{P}\left(Y(0)\neq Y\left(\alpha t^{\mathbb{K}}_{\text{rel}}\right)\right)\nonumber\\
	=&\sum_{i=0}^{1}\mathbb{P}\left(Y\left(\alpha t^{\mathbb{K}}_{\text{rel}}\right)\neq i|Y(0)=i\right)\mathbb{P}\left(Y(0)=i\right)\nonumber\\
	=&\sum_{i=0}^{1}\mathbb{P}\left(Y\left(\alpha t^{\mathbb{K}}_{\text{rel}}\right)=1-i|Y(0)=i\right)\mathbb{P}\left(Y(0)=i\right)\nonumber\\
	=&\sum_{i=0}^{1} \text{ The }(i,1-i)\text{th entry of }e^{\alpha t^{\mathbb{K}}_{\text{rel}}Q_n}\times\mathbb{P}\left(Y(0)=i\right)\nonumber\\
	=&\frac{|E_n|}{n}(1-e^{-\nu n \alpha t^{\mathbb{K}}_{\text{rel}}})\times\frac{|O_n|}{n}+\frac{|O_n|}{n}(1-e^{-\nu n \alpha t^{\mathbb{K}}_{\text{rel}}})\times\frac{|E_n|}{n},\quad\text{ because }X_{\mathbb{K}}(0)\sim U_n\nonumber\\
	=&\frac{1-\epsilon_n^2}{2}\times (1-e^{- \alpha }),\qquad\text{ using \eqref{eq:int_parity-1} and }t^{\mathbb{K}}_{\text{rel}}=\frac{1}{\nu n}.
\end{align}
Thus, we get
\begin{align*}
	&\sup_n \mathbb{P}\left(\phi_n\left(X_{\mathbb{K}}(0)\right)\neq \phi_n\left(X_{\mathbb{K}}\left(\alpha t^{\mathbb{K}}_{\text{rel}}\right)\right)\right)\\
	=&\sup_n \frac{1-\epsilon_n^2}{2}\times (1-e^{- \alpha })= \frac{1}{2}(1-e^{- \alpha })\longrightarrow 0\text{ as }\alpha\rightarrow 0.\qedhere
\end{align*}
\end{proof}
The next Boolean function is well studied in the combinatorics and the computer science literature (cf. \cite{EFF1,EFF2,EFF3}).
%%%%%%%%%%%
%%%%%%%%%%%
%%%%%%%%%%%
\begin{exmp}\label{ex:dictator}
	Fix $k\in [n]$, and $0<\varepsilon<1$. The dictator function $\phi^k_n:S_n\longrightarrow \{0,1\}$ given by 
	\[\phi^k_n:=\mathbb{1}_{\{\pi\in S_n:\pi(k)\leq\varepsilon n\}}\]
	is noise stable with respect to the interchange process on $K_n$.
\end{exmp}
\begin{proof}
	We first define two subsets of $[n]:=\{1,2,\dots,n\}$ as follows
	\[F_n:=\{i\in[n]:i\leq\varepsilon n\}\quad\text{and}\quad F^{\complement}_n:=\{i\in[n]:i>\varepsilon n\}.\]
	Therefore,
	\begin{equation}\label{eq:dictator-1}
		\frac{|F_n|}{n}=\frac{\lfloor\varepsilon n\rfloor}{n}\quad\text{and}\quad\frac{|F^{\complement}_n|}{n}=\frac{n-\lfloor\varepsilon n\rfloor}{n}.
	\end{equation}
	For $1\leq i<j\leq n$, let $T_{i,j}$ denote the exponential random variable representing the time when Poisson clock at edge $\{i,j\}$ rings. If $\sigma\in S_n$ with $\sigma(k)\leq\varepsilon n$, then the value of $\phi^k_n$ changes from $1$ to $0$ whenever the Poisson clock at the edge $\{k,\sigma^{-1}(a)\}$ rings for some $a\in F^{\complement}_n$, this happens at time 
	\[\min\{T_{k,\sigma^{-1}(a)}:a\in F^{\complement}_n\}\sim\text{exponential}\left(\nu|F^{\complement}_n|\right).\]
	Again, for $\sigma\in S_n$ with $\sigma(k)>\varepsilon n$, the value of $\phi^k_n$ changes from $0$ to $1$ whenever the Poisson clock at the edge $\{k,\sigma^{-1}(a)\}$ rings for some $a\in F_n$, this happens at time 
	\[\min\{T_{k,\sigma^{-1}(a)}:a\in F_n\}\sim\text{exponential}\left(\nu|F_n|\right).\]
	Thus, the process $Y(t):=\phi^k_n(X_{\mathbb K}(t))$ is itself a two-state continuous-time Markov chain, with infinitesimal generator
	\[
	Q_n
	=
	\begin{array}{c@{\quad}c}
		& \begin{array}{cc}
			0 &\;\qquad\; 1
		\end{array}
		\\[1mm]
		\begin{array}{c}
			0\\
			1
		\end{array}
		&
		\left(
		\begin{array}{cc}
			-\nu| F_n| & \nu| F_n|\\
			\nu| F^{\complement}_n| & -\nu| F^{\complement}_n|
		\end{array}
		\right)
	\end{array}.
	\]
	Here, the states are ordered as $0,1$. It follows that
	\begin{equation}
		e^{tQ_n}
		=
		\begin{array}{c@{\quad}c}
			& \begin{array}{cc}
				0 &\qquad\qquad\qquad\qquad 1
			\end{array}
			\\[1mm]
			\begin{array}{c}
				0\\[0.65cm]
				1
			\end{array}
			&
			\left(
			\begin{array}{cc}
				\dfrac{| F^{\complement}_n|}{n}
				+\dfrac{| F_n|}{n} e^{-\nu nt}
				
				&
				\dfrac{| F_n|}{n}
				\left(1-e^{-\nu nt}\right)
				\\[3mm]
				\dfrac{| F^{\complement}_n|}{n}
				\left(1-e^{-\nu nt}\right)
				&
				\dfrac{| F_n|}{n}
				+\dfrac{| F^{\complement}_n|}{n} e^{-\nu nt}
			\end{array}
			\right)
		\end{array},\qquad\text{ for all }t>0.
	\end{equation}
	Therefore, for $\alpha>0$, we have
	\begin{align}
		&\mathbb{P}\left(\phi^k_n\left(X_{\mathbb{K}}(0)\right)\neq \phi^k_n\left(X_{\mathbb{K}}\left(\alpha t^{\mathbb{K}}_{\text{rel}}\right)\right)\right)\nonumber\\
		=&\mathbb{P}\left(Y(0)\neq Y\left(\alpha t^{\mathbb{K}}_{\text{rel}}\right)\right)\nonumber\\
		=&\sum_{i=0}^{1}\mathbb{P}\left(Y\left(\alpha t^{\mathbb{K}}_{\text{rel}}\right)\neq i|Y(0)=i\right)\mathbb{P}\left(Y(0)=i\right)\nonumber\\
		=&\sum_{i=0}^{1}\mathbb{P}\left(Y\left(\alpha t^{\mathbb{K}}_{\text{rel}}\right)=1-i|Y(0)=i\right)\mathbb{P}\left(Y(0)=i\right)\nonumber\\
		=&\sum_{i=0}^{1} \text{ The }(i,1-i)\text{th entry of }e^{\alpha t^{\mathbb{K}}_{\text{rel}}Q_n}\times\mathbb{P}\left(Y(0)=i\right)\nonumber\\
		=&\frac{|F_n|}{n}(1-e^{-\nu n \alpha t^{\mathbb{K}}_{\text{rel}}})\times\frac{|F^{\complement}_n|}{n}+\frac{|F^{\complement}_n|}{n}(1-e^{-\nu n \alpha t^{\mathbb{K}}_{\text{rel}}})\times\frac{|F_n|}{n},\quad\text{ because }X_{\mathbb{K}}(0)\sim U_n\nonumber\\
		=&2\times\frac{\lfloor\varepsilon n\rfloor}{n}\times\left(1-\frac{\lfloor\varepsilon n\rfloor}{n}\right)\times(1-e^{- \alpha }),\qquad\text{ using \eqref{eq:dictator-1} and }t^{\mathbb{K}}_{\text{rel}}=\frac{1}{\nu n}.
	\end{align}
	Thus, we get
	\begin{align*}
		&\sup_n \mathbb{P}\left(\phi^k_n\left(X_{\mathbb{K}}(0)\right)\neq \phi^k_n\left(X_{\mathbb{K}}\left(\alpha t^{\mathbb{K}}_{\text{rel}}\right)\right)\right)\\
		=&\sup_n \;\;2\times\frac{\lfloor\varepsilon n\rfloor}{n}\times\left(1-\frac{\lfloor\varepsilon n\rfloor}{n}\right)\times(1-e^{- \alpha })\leq 2\varepsilon(1-e^{- \alpha })\longrightarrow 0\text{ as }\alpha\rightarrow 0.\qedhere
	\end{align*}
\end{proof}

We now turn to the star graph $\ST$. Recall that the star graph $\ST$ is the undirected graph with vertex set $[n]$ and edge set $\{\{i,n\}: 1 \leq i < n\}$. We use the notation $\{X_{\mathbb{S}}(t)\}_{t\geq 0}$ to denote the interchange process on $\ST$. In our case, all the Poisson clocks ring at the same constant rate $\nu'$, and thus the relaxation time is $t^{\mathbb{S}}_{\text{rel}}=\frac{1}{\nu'}$ (cf. the Aldous-Caputo-Liggett-Richthammer theorem \cite{CLR} and \cite{FOW}). Moreover, the uniform distribution on $S_n$, denoted by $U_n$, is the stationary distribution. Since $X_{\mathbb{S}}(0)\sim U_n$, it follows that $X_{\mathbb{S}}\left(\alpha t^{\mathbb{S}}_{\text{rel}}\right)\sim U_n$. The examples are given below.

\begin{exmp}\label{ex:dictator-star}
	Fix $k\in [n]$, and $0<\varepsilon<1$. The dictator function $\phi^k_n:S_n\longrightarrow \{0,1\}$ given by 
	\[\phi^k_n:=\mathbb{1}_{\{\pi\in S_n:\pi(k)\leq\varepsilon n\}}\]
	is noise sensitive with respect to the interchange process on $\ST$.
\end{exmp}
\begin{proof}
	Let us define two subsets of $[n]:=\{1,2,\dots,n\}$ as follows
	\begin{equation*}
		F_n:=\{i\in[n]:i\leq\varepsilon n\}%=\{1,\dots,\lfloor\varepsilon n\rfloor\}\;
		\quad\text{and}\quad F^{\complement}_n:=\{i\in[n]:i>\varepsilon n\}.%=\{\lfloor\varepsilon n\rfloor+1,\dots,n\}.
	\end{equation*}
	Let $\alpha>0$ be arbitrary. We observe that
	\begin{equation}\label{eq:dictator-star-1}
		\mathbb{E}\left(\phi^k_n\left(X_{\mathbb{S}}(0)\right)\right)=\frac{|F_n|}{n}=\frac{\lfloor\varepsilon n\rfloor}{n}\quad\text{and}\quad\mathbb{E}\left(\phi^k_n\left(X_{\mathbb{S}}\left(\alpha t^{\mathbb{S}}_{\text{rel}}\right)\right)\right)=\frac{|F_n|}{n}=\frac{\lfloor\varepsilon n\rfloor}{n}.
	\end{equation}
	
	We first work with the case $k=n$, i.e., $\phi^n_n$.
	
	For $1\leq i<j\leq n$, let $T_{i,j}$ denote the exponential random variable representing the time when Poisson clock at edge $\{i,j\}$ rings. If $\sigma\in S_n$ with $\sigma(n)\leq\varepsilon n$, then the value of $\phi^n_n$ changes from $1$ to $0$ whenever the Poisson clock at the edge $\{n,\sigma^{-1}(a)\}$ rings for some $a\in F^{\complement}_n$, this happens at time 
	\[\min\{T_{n,\sigma^{-1}(a)}:a\in F^{\complement}_n\}\sim\text{exponential}\left(\nu'|F^{\complement}_n|\right).\]
	Again, for $\sigma\in S_n$ with $\sigma(n)>\varepsilon n$, the value of $\phi^n_n$ changes from $0$ to $1$ whenever the Poisson clock at the edge $\{n,\sigma^{-1}(a)\}$ rings for some $a\in F_n$, this happens at time 
	\[\min\{T_{n,\sigma^{-1}(a)}:a\in F_n\}\sim\text{exponential}\left(\nu'|F_n|\right).\]
	Thus, the process $Y(t):=\phi^n_n(X_{\mathbb{S}}(t))$ is itself a two-state continuous-time Markov chain, with infinitesimal generator
	\[
	Q_n
	=
	\begin{array}{c@{\quad}c}
		& \begin{array}{cc}
			0 &\;\qquad\; 1
		\end{array}
		\\[1mm]
		\begin{array}{c}
			0\\
			1
		\end{array}
		&
		\left(
		\begin{array}{cc}
			-\nu'| F_n| & \nu'| F_n|\\
			\nu'| F^{\complement}_n| & -\nu'| F^{\complement}_n|
		\end{array}
		\right)
	\end{array}.
	\]
	Here, the states are ordered as $0,1$. It follows that
	\begin{equation*}
		e^{tQ_n}
		=
		\begin{array}{c@{\quad}c}
			& \begin{array}{cc}
				0 &\qquad\qquad\qquad\qquad 1
			\end{array}
			\\[1mm]
			\begin{array}{c}
				0\\[0.65cm]
				1
			\end{array}
			&
			\left(
			\begin{array}{cc}
				\dfrac{| F^{\complement}_n|}{n}
				+\dfrac{| F_n|}{n} e^{-\nu' nt}
				
				&
				\dfrac{| F_n|}{n}
				\left(1-e^{-\nu' nt}\right)
				\\[3mm]
				\dfrac{| F^{\complement}_n|}{n}
				\left(1-e^{-\nu' nt}\right)
				&
				\dfrac{| F_n|}{n}
				+\dfrac{| F^{\complement}_n|}{n} e^{-\nu' nt}
			\end{array}
			\right)
		\end{array},\qquad\text{ for all }t>0.
	\end{equation*}
	Therefore, we have
	\begin{align}\label{eq:dictator-star-2}
		&\mathbb{E}\left(\phi^n_n\left(X_{\mathbb{S}}(0)\right) \phi^n_n\left(X_{\mathbb{S}}\left(\alpha t^{\mathbb{S}}_{\text{rel}}\right)\right)\right)\nonumber\\
		=&\mathbb{E}\left(Y(0) Y\left(\alpha t^{\mathbb{S}}_{\text{rel}}\right)\right)\nonumber\\
		=&\mathbb{P}\left(Y(0)=1, Y\left(\alpha t^{\mathbb{S}}_{\text{rel}}\right)=1\right)\nonumber\\
		=&\mathbb{P}\left(Y\left(\alpha t^{\mathbb{S}}_{\text{rel}}\right)=1|Y(0)=1\right)\mathbb{P}\left(Y(0)=1\right)\nonumber\\
		=&\left(\frac{|F_n|}{n}+\frac{|F^{\complement}_n|}{n}e^{-\nu' n \alpha t^{\mathbb{S}}_{\text{rel}}}\right)\times\frac{|F_n|}{n},\quad\text{ because }X_{\mathbb{K}}(0)\sim U_n\nonumber\\
		=&\left(\frac{\lfloor\varepsilon n\rfloor}{n}+\frac{n-\lfloor\varepsilon n\rfloor}{n}e^{- n\alpha }\right)\times \frac{\lfloor\varepsilon n\rfloor}{n},\quad\text{ using \eqref{eq:dictator-star-1} and }t^{\mathbb{S}}_{\text{rel}}=\frac{1}{\nu'}.
	\end{align}
	Thus, we get
	\begin{align*}
		&\cov\left(\phi^n_n\left(X_{\mathbb{S}}(0)\right), \phi^n_n\left(X_{\mathbb{S}}\left(\alpha t^{\mathbb{S}}_{\text{rel}}\right)\right)\right)\\
		=&\mathbb{E}\left(\phi^n_n\left(X_{\mathbb{S}}(0)\right) \phi^n_n\left(X_{\mathbb{S}}\left(\alpha t^{\mathbb{S}}_{\text{rel}}\right)\right)\right)-\mathbb{E}\left(\phi^n_n\left(X_{\mathbb{S}}(0)\right)\right)\mathbb{E}\left(\phi^n_n\left(X_{\mathbb{S}}\left(\alpha t^{\mathbb{S}}_{\text{rel}}\right)\right)\right)\\
		=&\left(\frac{\lfloor\varepsilon n\rfloor}{n}+\frac{n-\lfloor\varepsilon n\rfloor}{n}e^{- n\alpha }\right)\times \frac{\lfloor\varepsilon n\rfloor}{n}-\left(\frac{\lfloor\varepsilon n\rfloor}{n}\right)^2\quad\text{ using \eqref{eq:dictator-star-2} and \eqref{eq:dictator-star-1}}\\
		=&\frac{\lfloor\varepsilon n\rfloor}{n}\times\left(1-\frac{\lfloor\varepsilon n\rfloor}{n}\right)e^{- n\alpha }\longrightarrow 
		% \varepsilon\times(1-\varepsilon)\times 0
		0\text{ as }n\rightarrow\infty.
	\end{align*}
	
	We now focus on the case $k\neq n$. Observe that $\phi^k_n(\pi)=\phi^n_n(\pi\cdot (k,n))$ for all $\pi\in S_n$. The interchange process on $\ST$ is the continuous-time random walk on $S_n$ generated by the set $\{(i,n):i=1,2,\dots,n-1\}$, each such transposition occurs at rate $\nu'$. Therefore, for all $\pi,\sigma\in S_n$,  we have
	\[\mathbb{P}\left(X_{\mathbb{S}}\left(\alpha t^{\mathbb{S}}_{\text{rel}}\right)=\sigma\cdot (k,n)|X_{\mathbb{S}}(0)=\pi\cdot (k,n)\right)=\mathbb{P}\left(X_{\mathbb{S}}\left(\alpha t^{\mathbb{S}}_{\text{rel}}\right)=\sigma|X_{\mathbb{S}}(0)=\pi\right).\]
	Moreover, $X_{\mathbb{S}}(0)\sim U_n$. Hence,
	\[\mathbb{P}\left(X_{\mathbb{S}}\left(\alpha t^{\mathbb{S}}_{\text{rel}}\right)=\sigma\cdot (k,n),X_{\mathbb{S}}(0)=\pi\cdot (k,n)\right)=\mathbb{P}\left(X_{\mathbb{S}}\left(\alpha t^{\mathbb{S}}_{\text{rel}}\right)=\sigma,X_{\mathbb{S}}(0)=\pi\right).\]
	Consequently,
	\[\cov\left(\phi^k_n\left(X_{\mathbb{S}}(0)\right), \phi^k_n\left(X_{\mathbb{S}}\left(\alpha t^{\mathbb{S}}_{\text{rel}}\right)\right)\right)=\cov\left(\phi^n_n\left(X_{\mathbb{S}}(0)\right), \phi^n_n\left(X_{\mathbb{S}}\left(\alpha t^{\mathbb{S}}_{\text{rel}}\right)\right)\right).\qedhere\]
\end{proof}

\section{Background on the symmetric group representation}\label{sec:Sn_repn}
In this section, we start with the necessary combinatorial backgrounds for describing the representation theory of the symmetric group $S_n$, then we recall the representation theory of $S_n$; in particular, the Vershik-Okounkov approach to the representation theory of $S_n$. We mainly followed the references \cite{Sagan,VO}, for this part. 
\subsection{Combinatorial frameworks for the representation theory of $S_n$}A \emph{partition} $\lambda$ of a positive integer $n$ is a weakly decreasing finite sequence $(\lambda_1,\cdots,\lambda_{\ell})$ of positive integers such that $\displaystyle\sum_{i=1}^{\ell}\lambda_i=n$. We write $\lambda\vdash n$ to mean $\lambda$ is a partition of $n$. The set of all partitions of $n$ is denoted by $\Par(n)$. The partition $\lambda$ can be pictorially visualised as a left-justified arrangement of $\ell$ rows of boxes with $\lambda_i$ boxes in the $i$th row (English notation). This pictorial arrangement of boxes is known as the \emph{Young diagram} of $\lambda$. For example there are five partitions of the positive integer $4$ viz. (4), (3,1), (2,2), (2,1,1) and (1,1,1,1). Young diagrams corresponding to the partitions of $4$ are given in Figure \ref{fig:ch2_yng_diag_with_4_boxes}. We use the same notation $\lambda$ to denote partition and Young diagram both. It will be clear from the context whether a partition or a Young diagram is meant. %The set of all Young diagrams (there is a unique Young diagram with zero boxes) is denoted by $\y$ and the set of all Young diagrams with $n$ boxes is denoted by $\yn$. 
\begin{figure}[h]
		\centering
		$\begin{array}{cclll}
		\yng(4)&\hspace{0.5cm}\yng(3,1)& \hspace{0.5cm}\yng(2,2) & \hspace{0.5cm}\yng(2,1,1) & \hspace{0.75cm}\yng(1,1,1,1)\\
			(4)\;&\;\quad(3,1)&\quad\;(2,2)&\quad(2,1,1)&\;(1,1,1,1)
		\end{array}$
		\caption{Young diagrams with $4$ boxes.}\label{fig:ch2_yng_diag_with_4_boxes}
\end{figure}
A \emph{Young tableau} of shape $\lambda$ (or \emph{$\lambda$-tableau}) is obtained by filling (bijectively) the numbers $1,\dots,n$ in the boxes of the Young diagram of $\lambda$. A $\lambda$-tableau is \emph{standard} if the entries in its boxes increase from left to right along rows and from top to bottom along columns. The set of all standard Young tableaux of a given shape $\lambda$ is denoted by $\std(\lambda)$ and the number of standard Young tableaux of shape $\lambda$ is denoted by $f^{\lambda}$. For example, all standard Young tableaux of shape $(3,1)$ are listed in Figure \ref{fig:ch2_yng_tab_of_shape_(3,1)}.
\begin{figure}[h]
		\centering
		$\begin{array}{ccc}
			\young({{\substack{1}}}{{\substack{2}}}{{\substack{3}}},{{\substack{4}}}) &\quad\young({{\substack{1}}}{{\substack{2}}}{{\substack{4}}},{{\substack{3}}}) & \quad \young({{\substack{1}}}{{\substack{3}}}{{\substack{4}}},{{\substack{2}}})
		\end{array}$
		\caption{Standard Young tableaux of shape $(3,1)$.}\label{fig:ch2_yng_tab_of_shape_(3,1)}
\end{figure}
The \emph{content} of a box in row $i$ and column $j$ of a Young diagram is the integer $j-i$. The \emph{diagonal index} of a partition $\lambda\vdash n$, denoted Diag$(\lambda)$, defined as the sum of the contents of all boxes in $\lambda$, i.e., \[\text{Diag}(\lambda)=\displaystyle\sum_{\substack{\text{the box at row }i,\\\text{ column }j\text{ is present in }\lambda}}(j-i).\]
Given $T\in\std(\lambda)$, let $b_T(u)$ denote the box in $T$, where the number $u$ resides. Also $c(b_T(u))$ denotes the content of the box $b_T(u)$, $1\leq u\leq n$. Given a Young diagram $\lambda$, its \emph{conjugate} $\lambda^{\prime}$ is obtained by reflecting $\lambda$ with respect to the diagonal consisting of boxes with content $0$. A Young diagram $\lambda$ is \emph{self-conjugate} if $\lambda^{\prime}=\lambda$. We are now in a position to discuss the representation theory of the symmetric group $S_n$.
\subsection{Representation theory of $S_n$ and orthonormal basis for $\mathbb{C}^{S_n}$}
Recall that $S_n$ is the group of all bijections on $[n]$. Elements of $S_n$ are known as \emph{permutations} and the group operation is the composition of permutations. Unless otherwise stated the field of scalars is considered to be $\mathbb{C}$. \emph{The irreducible representations of $S_n$ are indexed by the partitions of $n$} \cite[Theorem 2.4.6]{Sagan}. The irreducible $S_n$-modules are called the \emph{Specht modules}. The Specht module indexed by $\lambda\vdash n$ is denoted by $V^{\lambda}$, and the associated irreducible character is denoted by $\chi^{\lambda}$. \emph{The basis vectors of $V^{\lambda}$ are indexed by the standard Young tableaux of shape $\lambda$} \cite[Theorem 2.5.2]{Sagan}, and hence $\dim(V^{\lambda})=f^{\lambda}$. Below we describe a scheme to obtain an orthonormal basis of $\mathbb{C}^{S_n}$ (of all complex valued functions defined on $S_n$).

If we have orthonormal bases $\mathcal{B}_{\lambda}$ of $V^{\lambda}$ for every $\lambda\vdash n$, then Lemma \ref{lem:ONB_for_G} provides that
	\begin{equation}\label{eq:ONB_for_Sn}
	\underset{\lambda\in\widehat{S_n}}{\bigcup}\big\{\psi^{ST}_{\mathcal{B}_{\lambda}}:S,T\in\std(\lambda)\big\}
\end{equation}
forms an orthonormal basis for the space of all complex valued functions defined on $S_n$. If $(\rho_{\lambda},V^{\lambda})$ denotes the irreducible representation of $S_n$ indexed by $\lambda$, then recall, from \eqref{eq:ONB_for_G}, that the functions $\psi^{ST}_{\mathcal{B}_{\lambda}}: S_n\rightarrow\mathbb{R}$ are defined by \[\psi^{ST}_{\mathcal{B}_{\lambda}}(\pi):=\sqrt{f^{\lambda}}\Tr\left(\bigg[\rho_{\lambda}(\pi^{-1})\bigg]_{\mathcal{B}_{\lambda}}E^{ST}_{\mathcal{B}_{\lambda}}\right),\;\pi\in S_n.\]
Here, $E^{ST}_{\mathcal{B}_{\lambda}}$ is the $f^{\lambda}\times f^{\lambda}$ matrix with rows and column indexed by the standard Young tableaux of shape $\lambda$ having only non-zero entry $1$ at row $S\in\std(\lambda)$ and column $T\in\std(\lambda)$.

We now see a specific orthonormal basis for $\mathbb{C}^{S_n}$ viz. the \emph{Gelfand-Tsetlin} basis (GT-basis). We start with recalling the branching rule for the symmetric group. Let $\lambda^{-}$ be the collection of all partitions of $n-1$ obtained by removing a single \emph{inner corner} (a corner box in $\lambda$ whose removal leaves a valid Young diagram of $n-1$) from the Young diagram of $\lambda\vdash n$. Then the precise statement of the branching rule is the following \cite[Theorem 2.8.3]{Sagan}:
\[V^{\lambda}\downarrow^{S_n}_{S_{n-1}}=\underset{\mu\in\lambda^{-}}{\oplus}V^{\mu}.\]
For example taking $n=15$, the restriction of $V^{(5,4,4,2)}$ to $S_{14}$ is given by the following:
\[V^{\tiny\yng(5,4,4,2)}\big\downarrow_{S_{14}}^{S_{15}}=V^{\tiny\yng(4,4,4,2)}\oplus V^{\tiny\yng(5,4,3,2)}\oplus V^{\tiny\yng(5,4,4,1)}.\]

We now state a useful lemma that will be used later.
\begin{lem}[{Frobenius, Vershik-Okounkov}]\label{lem:Sym_char}
	For $\lambda\vdash n$, let $\chi^{\lambda}$ denote the irreducible character of $S_n$ indexed by $\lambda$. Then, $\chi^{\lambda}((1,2))=\frac{f^{\lambda}}{\binom{n}{2}}\emph{Diag}(\lambda)$.
\end{lem}
\begin{proof}
	Let $\lambda'$ denote the conjugate partition of $\lambda$. Then, Fact 2 in \cite[p. 40]{D2} states that
	\[\frac{\chi^{\lambda}((1,2))}{f^{\lambda}}=\frac{1}{\binom{n}{2}}\sum_j\left(\binom{\lambda_j}{2}-\binom{\lambda'_j}{2}\right),\]
	where $\lambda_j$ denotes the $j$th part of $\lambda$, and $\lambda'_j$ denotes the $j$th part of $\lambda'$. Thus, the lemma follows from the fact that
	\[\sum_j\left(\binom{\lambda_j}{2}-\binom{\lambda'_j}{2}\right)=\text{Diag}(\lambda).\]
\end{proof}
\section{Proof of Theorem \ref{thm:gen<->complete}}\label{sec:proof_of_thm:1.1}
In this section, We focus on the noise sensitivity with respect to interchange processes on $K_n$ and $\ST$. We will prove Theorem \ref{thm:gen<->complete} in this section.
\subsection{Noise sensitivity with respect to the interchange process on $K_n$} We now consider the interchange process on the complete graph, with every edge rings with a constant rate $\nu>0$. In this section, we denote the set of all transpositions of $S_n$ by the notation $\mathcal{T}_n$,
\[\text{i.e., }\quad\mathcal{T}_n:=\{(i,j):1\leq i<j\leq n\}.\]
The associated rate function $\mu_{\mathbb{K}}:S_n\rightarrow (0,\infty)$ is defined by
\[\mu_{\mathbb{K}}(\pi)=\begin{cases}
	\nu&\text{ if }\pi\in \mathcal{T}_n,\\
	0&\text{ otherwise.}\
\end{cases}\]
and the infinitesimal generator is given by the matrix of the operator
\[-\Delta_{\mathbb{K}}(R):=\nu\displaystyle\sum_{\tau\in \mathcal{T}_n}(R_{\tau}-I)\]
on $\mathbb{C}^{S_n}$ with respect to the basis $\{\mathbb{1}_{\{\pi\}}:\pi\in S_n\}$. Here, $R$ is the regular representation of $S_n$ and $I$ is the identity operator. We now prove the following useful results.
\begin{lem}[{\cite[Corollary 1 and Lemma 7]{DS}}]\label{lem:compgrp_eigval_on_irr}
	For any $\lambda\vdash n$, let $(\rho_{\lambda},V^{\lambda})$ denote the irreducible representation of $S_n$ indexed by $\lambda$. Then,
	\[
		\Delta_{\mathbb{K}}(\rho_{\lambda}):=\nu\displaystyle\sum_{\tau\in \mathcal{T}_n}(I-\rho_{\lambda}(\tau))=\nu\left(\binom{n}{2}-\emph{Diag}(\lambda)\right) I,\quad \text{i.e., }\displaystyle\sum_{\tau\in \mathcal{T}_n}\rho_{\lambda}(\tau)=\emph{Diag}(\lambda) I.
	\]
\end{lem}
\begin{proof}
	We first note that 
	\begin{align*}
		&\rho_{\lambda}(\pi)\circ\left(\nu\displaystyle\sum_{\tau\in \mathcal{T}_n}(I-\rho_{\lambda}(\tau))\right)\circ\rho_{\lambda}(\pi^{-1})=\nu\displaystyle\sum_{\tau\in \mathcal{T}_n}(I-\rho_{\lambda}(\tau)),\\
		\text{i.e.,}\;\;&\rho_{\lambda}(\pi)\circ\Delta_{\mathbb{K}}(\rho_{\lambda})=\Delta_{\mathbb{K}}(\rho_{\lambda})\circ\rho_{\lambda}(\pi)
	\end{align*}
	 for all $\pi\in S_n$. Thus, Schur's lemma guarantees that the operator $\Delta_{\mathbb{K}}(\rho_{\lambda})$ is scalar; let $\Delta_{\mathbb{K}}(\rho_{\lambda})=cI$. Then, equating the trace, we get
	\[\nu\displaystyle\sum_{\tau\in \mathcal{T}_n}(f^{\lambda}-\chi^{\lambda}(\tau))=cf^{\lambda}.\]
	Using the fact that $\chi^{\lambda}$ is a class function, we obtain
	\[
	\nu\binom{n}{2}\left(f^{\lambda}-\chi^{\lambda}(1,2)\right)=cf^{\lambda},\;\text{ i.e., }\;c=\nu\left(\binom{n}{2}-\frac{\binom{n}{2}}{f^{\lambda}}\chi^{\lambda}(1,2)\right),
	\]
	Finally, using Lemma \ref{lem:Sym_char} we obtain $c=\nu\left(\binom{n}{2}-\text{Diag}(\lambda)\right)$. This complete the proof.
\end{proof}
\begin{lem}\label{lem:compgrp_eigval}
	Let $\lambda\vdash n$ and $S,T\in\std(\lambda)$. For any orthonormal basis $\mathcal{B}_{\lambda}$ of $V^{\lambda}$, recall that $\psi^{ST}_{\mathcal{B}_{\lambda}}(\pi):=\sqrt{f^{\lambda}}\Tr\left(\big[\rho_{\lambda}(\pi^{-1})\big]_{\mathcal{B}_{\lambda}}E^{ST}_{\mathcal{B}_{\lambda}}\right)$ for all $\pi\in S_n$. Then, the orthonormal basis 
	\[	\underset{\lambda\vdash n}{\bigcup}\big\{\psi^{ST}_{\mathcal{B}_{\lambda}}:S,T\in\std(\lambda)\big\}\]
	given in \eqref{eq:ONB_for_Sn} is an eigenbasis of $\Delta_{\mathbb{K}}(R)$. Moreover, the eigenvalue associated to the eigenvector $\psi^{ST}_{\mathcal{B}_{\lambda}}$ is $\nu\left(\binom{n}{2}-\emph{Diag}(\lambda)\right),\;S,T\in\std(\lambda)$ for every $\lambda\vdash n$.
\end{lem}
\begin{proof}
	For every $\lambda\vdash n$ and $S,T\in\std(\lambda)$; we have
	\[\Delta_{\mathbb{K}}(R)\psi^{ST}_{\mathcal{B}_{\lambda}}=\nu\displaystyle\sum_{\tau\in \mathcal{T}_n}(I-R_{\tau})\psi^{ST}_{\mathcal{B}_{\lambda}}=\nu\displaystyle\sum_{\tau\in \mathcal{T}_n}(\psi^{ST}_{\mathcal{B}_{\lambda}}-R_{\tau}\psi^{ST}_{\mathcal{B}_{\lambda}})\]
	Therefore, for every $\pi\in S_n$, we get
	\begin{align*}
		\Delta_{\mathbb{K}}(R)\psi^{ST}_{\mathcal{B}_{\lambda}}(\pi)&=\nu\displaystyle\sum_{\tau\in \mathcal{T}_n}\left(\psi^{ST}_{\mathcal{B}_{\lambda}}(\pi)-R_{\tau}\psi^{ST}_{\mathcal{B}_{\lambda}}(\pi)\right)\\
		&=\nu\left(\binom{n}{2}\psi^{ST}_{\mathcal{B}_{\lambda}}(\pi)-\displaystyle\sum_{\tau\in \mathcal{T}_n}R_{\tau}\psi^{ST}_{\mathcal{B}_{\lambda}}(\pi)\right)\\
		&=\nu\left(\binom{n}{2}\psi^{ST}_{\mathcal{B}_{\lambda}}(\pi)-\displaystyle\sum_{\tau\in \mathcal{T}_n}\psi^{ST}_{\mathcal{B}_{\lambda}}(\pi\tau)\right)\\
		&=\nu\left(\binom{n}{2}\psi^{ST}_{\mathcal{B}_{\lambda}}(\pi)-\displaystyle\sum_{\tau\in \mathcal{T}_n}\sqrt{f^{\lambda}}\Tr\left(\bigg[\rho_{\lambda}(\tau^{-1}\pi^{-1})\bigg]_{\mathcal{B}_{\lambda}}E^{ST}_{\mathcal{B}_{\lambda}}\right)\right)
	\end{align*}
	Now, using the linearity of trace; we obtain
	\begin{align*}
		\Delta_{\mathbb{K}}(R)\psi^{ST}_{\mathcal{B}_{\lambda}}(\pi)
		&=\nu\left(\binom{n}{2}\psi^{ST}_{\mathcal{B}_{\lambda}}(\pi)-\sqrt{f^{\lambda}}\Tr\left(\displaystyle\sum_{\tau\in \mathcal{T}_n}\bigg[\rho_{\lambda}(\tau^{-1}\pi^{-1})\bigg]_{\mathcal{B}_{\lambda}}E^{ST}_{\mathcal{B}_{\lambda}}\right)\right)\\
		&=\nu\left(\binom{n}{2}\psi^{ST}_{\mathcal{B}_{\lambda}}(\pi)-\sqrt{f^{\lambda}}\Tr\left(\displaystyle\sum_{\tau\in \mathcal{T}_n}\bigg[\rho_{\lambda}(\tau^{-1})\bigg]_{\mathcal{B}_{\lambda}}\bigg[\rho_{\lambda}(\pi^{-1})\bigg]_{\mathcal{B}_{\lambda}}E^{ST}_{\mathcal{B}_{\lambda}}\right)\right)\\
		&=\nu\left(\binom{n}{2}\psi^{ST}_{\mathcal{B}_{\lambda}}(\pi)-\sqrt{f^{\lambda}}\Tr\left(\bigg[\displaystyle\sum_{\tau\in \mathcal{T}_n}\rho_{\lambda}(\tau)\bigg]_{\mathcal{B}_{\lambda}}\bigg[\rho_{\lambda}(\pi^{-1})\bigg]_{\mathcal{B}_{\lambda}}E^{ST}_{\mathcal{B}_{\lambda}}\right)\right)\\
		&=\nu\left(\binom{n}{2}\psi^{ST}_{\mathcal{B}_{\lambda}}(\pi)-\sqrt{f^{\lambda}}\Tr\left(\text{Diag}(\lambda)\bigg[\rho_{\lambda}(\pi^{-1})\bigg]_{\mathcal{B}_{\lambda}}E^{ST}_{\mathcal{B}_{\lambda}}\right)\right),\text{ by Lemma \ref{lem:compgrp_eigval_on_irr}}\\
		&=\nu\left(\binom{n}{2}\psi^{ST}_{\mathcal{B}_{\lambda}}(\pi)-\text{Diag}(\lambda)\sqrt{f^{\lambda}}\Tr\left(\bigg[\rho_{\lambda}(\pi^{-1})\bigg]_{\mathcal{B}_{\lambda}}E^{ST}_{\mathcal{B}_{\lambda}}\right)\right)\\
		&=\nu\left(\binom{n}{2}\psi^{ST}_{\mathcal{B}_{\lambda}}(\pi)-\text{Diag}(\lambda)\psi^{ST}_{\mathcal{B}_{\lambda}}(\pi)\right)=\nu\left(\binom{n}{2}-\text{Diag}(\lambda)\right)\psi^{ST}_{\mathcal{B}_{\lambda}}(\pi).
	\end{align*}
	Thus, we get $\Delta_{\mathbb{K}}(R)\psi^{ST}_{\mathcal{B}_{\lambda}}=\nu\left(\binom{n}{2}-\text{Diag}(\lambda)\right)\psi^{ST}_{\mathcal{B}_{\lambda}}$ for every $S,T\in\std(\lambda),\lambda\vdash n$. Therefore, the lemma follows.
\end{proof}
\begin{lem}\label{lem:compgrp_boolfn_projec}
	Let $\lambda\vdash n$ and $S,T\in\std(\lambda)$. For any orthonormal basis $\mathcal{B}_{\lambda}$ of $V^{\lambda}$, recall that $\psi^{ST}_{\mathcal{B}_{\lambda}}(\pi):=\sqrt{f^{\lambda}}\Tr\left(\big[\rho_{\lambda}(\pi^{-1})\big]_{\mathcal{B}_{\lambda}}E^{ST}_{\mathcal{B}_{\lambda}}\right)$ for all $\pi\in S_n$. Then, for any Boolean function $\f:S_n\rightarrow \{0,1\}$, we have
	\[\langle\f,\psi^{ST}_{\mathcal{B}_{\lambda}}\rangle=\frac{\sqrt{f^{\lambda}}}{n!}\times\text{The }(S,T)\text{th entry of the matrix }\Big[\widehat{\f}(\lambda)\Big]_{\mathcal{B}_{\lambda}}.\]
	Here, the $(S,T)$th entry means the element at row $S$ and column $T$.
\end{lem}
\begin{proof}
	We first note that $\widehat{\psi}^{ST}_{\mathcal{B}_{\lambda}}(\eta)=\delta_{\eta,\lambda}\frac{n!}{\sqrt{f^{\lambda}}}E^{ST}_{\mathcal{B}_{\lambda}}$ for all $\eta\vdash n$, this follows from \eqref{eq:FT_of_ONB} by setting $G=S_n$. Therefore, we have
		\[\langle\psi^{ST}_{\mathcal{B}_{\lambda}},\f\rangle=\sum_{\pi\in S_n}\psi^{ST}_{\mathcal{B}_{\lambda}}(\pi)\overline{\f(\pi)}\;\frac{1}{n!}=\frac{1}{n!}\sum_{\pi\in S_n}\psi^{ST}_{\mathcal{B}_{\lambda}}(\pi)\check{\f}(\pi^{-1})\]
where $\check{\f}(\pi)=\overline{\f(\pi^{-1})}$ for all $\pi\in S_n$. Then, we have
	\[
	\Big[\widehat{\check{\f}}(\lambda)\Big]_{\mathcal{B}_{\lambda}}:=\Big[\widehat{\f}(\lambda)\Big]_{\mathcal{B}_{\lambda}}^{\dagger}\text{, the conjugate-transpose  of the matrix }\Big[\widehat{\f}(\lambda)\Big]_{\mathcal{B}_{\lambda}}.
\]
This follows from the fact that $\big[\rho(g)\big]_{\mathcal{B}^{\rho}}$ is unitary matrix for every $g\in G$. Therefore, using the Plancherel formula \eqref{eq:Plancherel formula}, we get
\begin{align*}
	\langle\psi^{ST}_{\mathcal{B}_{\lambda}},\f\rangle=\frac{1}{n!}\sum_{\pi\in S_n}\check{\f}(\pi^{-1})\psi^{ST}_{\mathcal{B}_{\lambda}}(\pi)&=\frac{1}{n!^2}\sum_{\eta\vdash n}f^{\eta}\Tr\left(\widehat{\check{\f}}(\eta)\widehat{\psi}^{ST}_{\mathcal{B}_{\lambda}}(\eta)\right)\\
	&=\frac{1}{n!^2}\sum_{\eta\vdash n}f^{\eta}\Tr\left(\widehat{\check{\f}}(\eta)\delta_{\eta,\lambda}\frac{n!}{\sqrt{f^{\lambda}}}E^{ST}_{\mathcal{B}_{\lambda}}\right)\\
	&=\frac{f^{\lambda}}{n!^2}\Tr\left(\widehat{\check{\f}}(\lambda)\frac{n!}{\sqrt{f^{\lambda}}}E^{ST}_{\mathcal{B}_{\lambda}}\right)\\
	&=\frac{\sqrt{f^{\lambda}}}{n!}\Tr\left(\Big[\widehat{\f}(\lambda)\Big]_{\mathcal{B}_{\lambda}}^{\dagger}E^{ST}_{\mathcal{B}_{\lambda}}\right)\\
	&=\frac{\sqrt{f^{\lambda}}}{n!}\times\overline{\text{The }(S,T)\text{th entry of the matrix }\Big[\widehat{\f}(\lambda)\Big]_{\mathcal{B}_{\lambda}}}.\qedhere
\end{align*}
\end{proof}
Lemma \ref{lem:compgrp_eigval} provides that the smallest eigenvalue is $0$, and it comes for the partition $\lambda=(n)$. Also, the smallest positive eigenvalue is $\nu n$, and it arrives from the partition $\lambda=(n-1,1)$. Therefore, for any positive integer $k$, we have the following
\begin{align*}
	\nu n\leq \nu\left(\binom{n}{2}-\text{Diag}(\lambda)\right)< k\nu n\;\text{ if and only if }\;\binom{n}{2}-kn<\text{Diag}(\lambda)\leq\frac{n(n-3)}{2}.
\end{align*}
From now on, we set the following notation,
\begin{equation}\label{eq:RT-partition_range}
	\Lambda_k:=\Bigg\{\lambda\vdash n: {n\choose 2}-nk<\text{Diag}(\lambda)\leq\frac{n(n-3)}{2}\Bigg\}.
\end{equation}
Therefore, Lemma \ref{lem:compgrp_boolfn_projec}, Lemma \ref{lem:spectral_NS}, and Lemma \ref{lem:spectral_NStable} implies the following theorem:
\begin{thm}\label{thm:compgrp_NS-NStability}
	Let $\lambda\vdash n$, and $\mathcal{B}_{\lambda}$ be any orthonormal basis of $V^{\lambda}$. Also, for every $n\geq 1$, let $\f_n:S_n\longrightarrow\{0,1\}$ be a Boolean function. Then,
	\begin{enumerate}
		\item The sequence of $\{\f_n\}_n$ is noise sensitive with respect to the interchange process on $K_n$ if and only if for all positive integer $k>0$
		\[	\lim_{n\rightarrow\infty}\;\sum_{\lambda\in\Lambda_k} \frac{f^{\lambda}}{n!^2}\sum_{S,T\in\std(\lambda)}\left|\text{The }(S,T)\text{th entry of the matrix }\Big[\widehat{\f_n}(\lambda)\Big]_{\mathcal{B}_{\lambda}}\right|^2=0.\]
		\item The sequence of $\{\f_n\}_n$ is noise stable with respect to the interchange process on $K_n$ if and only if for all $\delta>0$ there exists a positive integer $k$ such that
		\[\sup_n\;\sum_{\substack{\lambda\vdash n\\
				\lambda\notin \Lambda_k\cup\{(n)\}}} \frac{f^{\lambda}}{n!^2}\sum_{S,T\in\std(\lambda)}\left|\text{The }(S,T)\text{th entry of the matrix }\Big[\widehat{\f_n}(\lambda)\Big]_{\mathcal{B}_{\lambda}}\right|^2<\delta.\]
	\end{enumerate}
\end{thm}
This theorem provides a criterion for testing the noise sensitivity or noise stability with respect to the interchange process on $K_n$. 
\begin{cor}\label{cor:compgrp_NS}
	Let $\lambda\vdash n$, and $\mathcal{B}_{\lambda}$ be any orthonormal basis of $V^{\lambda}$. Also, for every $n\geq 1$, let $\f_n:S_n\longrightarrow\{0,1\}$ be a Boolean function. Then, The sequence of $\{\f_n\}_n$ is noise sensitive with respect to the interchange process on $K_n$ if and only if for all positive integer $k>0$
		\[	\lim_{n\rightarrow\infty}\;\sum_{r=1}^{k}\sum_{\xi\vdash r} \frac{f^{(n-r,\xi)}}{n!^2}\hspace{-0.5cm}\sum_{S,T\in\std((n-r,\xi))}\left|\text{The }(S,T)\text{th entry of the matrix }\Big[\widehat{\f_n}((n-r,\xi))\Big]_{\mathcal{B}_{(n-r,\xi)}}\right|^2=0.\]
\end{cor}
\begin{proof}
	For any positive integer $r$ satisfying $1\leq r\leq\frac{n}{2}$, we have
	\begin{align}\label{eq:long_1st-part_partition}
		&\text{Diag}(n-r,1^r)\leq\text{Diag}(n-r,\xi)\leq\text{Diag}(n-r,r)\nonumber\\
		\text{i.e, }\quad&\binom{n}{2}-rn\leq\text{Diag}(n-r,\xi)\leq\binom{n}{2}-rn+r^2-r,
	\end{align}
	for all $\xi\vdash r$. Thus, given an arbitrary integer $k$, and $k<r\leq\frac{n}{2}$, we have
	\[\binom{n}{2}-rn+r^2-r<\binom{n}{2}-kn,\text{ for large enough }n.\]
	Therefore, given an arbitrary integer $k$, in view of \eqref{eq:RT-partition_range} and \eqref{eq:long_1st-part_partition}, we can conclude that
	\begin{equation}\label{eq:RT-partition_range-decomposition}
		\Lambda_k=\bigcup_{r=1}^{k}\bigg\{(n-r,\xi)\vdash n:\xi\vdash r\bigg\},\text{ for sufficiently large }n.
	\end{equation}
	Thus, the corollary follows from the first part of Theorem \ref{thm:compgrp_NS-NStability}.
\end{proof}
\subsection{Noise sensitivity with respect to the weighted interchange process on $\mathscr{G}$}\label{subsec:star_graph}

In this section, we work with the interchange process on $\G$. Here, $\G$ is connected graph with vertex set $[n]:=\{1,\dots,n\}$ and edge set $\mathcal{E}$. For the interchange process on $\G$, we consider that the Poisson clock at edge $\{i,j\}\in \mathcal{E}$ rings with a constant rate $x_{i,j}>0$.
The associated rate function $\mu_{\mathscr{G}}:S_n\rightarrow (0,\infty)$ is defined by
\[\mu_{\G}(\pi)=\begin{cases}
	x_{i,j}&\text{ if }\pi=(i,j)\text{ for }\{i,j\}\in\mathcal{E},\\
	0&\text{ otherwise.}\
\end{cases}\]
and the infinitesimal generator is given by the matrix of the operator
\[-\Delta_{\G}(R):=\displaystyle\sum_{\{i,j\}\in\mathcal{E}}x_{i,j}(R_{(i,j)}-I)\]
on $\mathbb{C}^{S_n}$ with respect to the basis $\{\mathbb{1}_{\{\pi\}}:\pi\in S_n\}$. Here, $R$ is the regular representation of $S_{n}$ and $I$ is the identity operator. We now compare it with the interchange process on $K_n$, with infinitesimal generator
\[-\Delta_{\mathbb{K}}(R):=\nu\displaystyle\sum_{1\leq i<j\leq n}(R_{(i,j)}-I).\]
First, note that both interchange processes are continuous-time reversible Markov chains on $S_n$. Moreover, $\Delta_{\mathbb{K}}(R)\Delta_{\G}(R)=\Delta_{\G}(R)\Delta_{\mathbb{K}}(R)$. Since both operators are self-adjoint and commute, there exists an orthonormal basis of $\mathbb{C}^{S_n}$ consisting of common eigenvectors of $\Delta_{\G}(R)$ and $\Delta_{\mathbb{K}}(R)$. 

For every $\lambda\vdash n$, let $(\rho_{\lambda},V^{\lambda})$ denote the irreducible representation of $S_n$ indexed by $\lambda$, and define
\[\Delta_{\mathbb{K}}(\rho_{\lambda}):=\nu\displaystyle\sum_{1\leq i<j\leq n}(I-\rho_{\lambda}((i,j)))\quad\text{ and }\quad\Delta_{\G}(\rho_{\lambda}):=\displaystyle\sum_{\{i,j\}\in\mathcal{E}}x_{i,j}(I-\rho_{\lambda}((i,j))).\]
Recalling that $\mathbb{C}^{S_n}\cong\underset{\lambda\vdash n}{\oplus}f^{\lambda}V^{\lambda}$ as $S_n$-modules (see \eqref{eq:Group_alg._decom.}), choose a basis with respect to which
\begin{equation}\label{eq:torus<->complete-laplacian_decomposition}
	[\Delta_{\mathbb{K}}(R)]=\underset{\lambda\vdash n}{\bigoplus}\left([I]_{\lambda}\bigotimes\big[\Delta_{\mathbb{K}}(\rho_{\lambda})\big]\right)\quad\text{ and }\quad[\Delta_{\G}(R)]=\underset{\lambda\vdash n}{\bigoplus}\left([I]_{\lambda}\bigotimes\big[\Delta_{\G}(\rho_{\lambda})\big]\right)
\end{equation}
in matrix notation. Here, $\bigoplus$ denotes the direct sum of matrices, $\bigotimes$ denotes the Kronecker product of matrices, and $[I]_{\lambda}$ denotes the identity matrix of size $f^{\lambda}\times f^{\lambda}$. Consequently, the eigenvalues of $\Delta_{\G}(R)$ (respectively, $\Delta_{\mathbb{K}}(R)$) are obtained from the eigenvalues of $\Delta_{\G}(\rho_{\lambda})$ (respectively, $\Delta_{\mathbb{K}}(\rho_{\lambda})$) for all $\lambda\vdash N$, each counted $f^{\lambda}$ times\footnote{When the same eigenvalue occurs for different partitions, its multiplicities are summed.}. For every $\lambda\vdash n$ and $T\in\std(\lambda)$, let $\theta_{\G}^{\lambda,T}$ (respectively,  $\theta_{\mathbb{K}}^{\lambda,T}$) denote the corresponding eigenvalue of $\Delta_{\G}(\rho_{\lambda})$ (respectively, 
$\Delta_{\mathbb{K}}(\rho_{\lambda})$). Since $V^{\lambda}$ appears with multiplicity $f^{\lambda}$ in $\mathbb{C}^{S_n}$, we may choose common orthonormal eigenvectors $\{\psi^{ST}\in\mathbb{C}^{S_n}:S\in\std(\lambda)\}$ associated with the pair of eigenvalues $\bigl(\theta_{\G}^{\lambda,T}, \theta_{\mathbb{K}}^{\lambda,T}\bigr)$. From now on, we work with the common orthonormal eigenbasis
\begin{equation}\label{eq:torus<->complete-ONB}
	\underset{\lambda\vdash n}{\bigcup}\{\psi^{ST}\in\mathbb{C}^{S_n}:T,S\in\std(\lambda)\}\qquad\text{of}\qquad\mathbb{C}^{S_n}.
\end{equation}
%%%%
%%%%
%%%%
%%%%

%Lemma \ref{lem:compgrp_eigval} shows that the smallest positive eigenvalue of $\Delta_{\mathbb{K}}(R)$ is $\nu N$. 
The smallest positive eigenvalue of $\Delta_{\G}(R)$ is, by the Aldous–Caputo–Liggett–Richthammer theorem \cite[Theorem 1.1]{CLR}, equal to the smallest eigenvalue of $\Delta_{\G}(\rho_{(n-1,1)})$. We denote this eigenvalue by
\[\theta_{\G}:=\displaystyle\min_{T\in\std((n-1,1))}\theta_{\G}^{(n-1,1),T}.\]
We now recall some useful results from a recent work of Gil Alon \cite{Alon_Oct26}. We thank him for bringing his results to our attention and pointing out the relevant theorem and proposition.
\begin{lem}[{\cite[Theorem 1.7]{Alon_Oct26}}]\label{lem:comp<->gen-laplacian_comparison1}
	If $\tau\vdash n$ has first row $\tau_1=n-r$, where $1\leq r\leq \lfloor \frac{n}{2}\rfloor$, then
	\[\displaystyle\min_{T\in\std(\tau)}\theta_{\G}^{\tau,T}=\text{the minimum eigenvalue of }\Delta_{\G}(\rho_{\tau})\geq \frac{1}{6\log 3}\;r \theta_{\G}.\]
\end{lem}
\begin{lem}[{\cite[Proposition 5.9]{Alon_Oct26}}]\label{lem:comp<->gen-laplacian_comparison2}
	Let $n\geq 2$. If $\tau\vdash n$ has first row $\tau_1=n-r$, where $\lfloor \frac{n}{2}\rfloor+1\leq r\leq n-1$, then
	\[\displaystyle\min_{T\in\std(\tau)}\theta_{\G}^{\tau,T}=\text{the minimum eigenvalue of }\Delta_{\G}(\rho_{\tau})\geq \frac{r}{8\log n}\; \theta_{\G}.\]
\end{lem}
We are now in a position to prove Theorem \ref{thm:gen<->complete}, the proof is given below.
\begin{proof}[Proof of Theorem \ref{thm:gen<->complete}]
	For any positive integer $k$, recalling $\Lambda_k$ from \eqref{eq:RT-partition_range-decomposition}, we have
	\begin{equation}\label{eq:gen<->complete1}
		\Lambda_{\lceil 6\log 3\rceil k-1}^{\complement}\setminus\{(n)\}=\{\tau\vdash n: \tau_1<n-\lceil 6\log 3\rceil k+1\}.
	\end{equation}
	For $\tau\vdash n$ satisfying $\tau_1<n-\lceil 6\log 3\rceil k+1$, we have $\tau_1=n-r$ where 
	\[r\in\big\{\lceil 6\log 3\rceil k,\lceil 6\log 3\rceil k+1,\dots,n-1\big\}.\]
	Therefore, by Lemma \ref{lem:comp<->gen-laplacian_comparison2} and Lemma \ref{lem:comp<->gen-laplacian_comparison2}, for $n\geq 2$, we have
	\begin{align}\label{eq:gen<->complete2}
		\displaystyle\min_{T\in\std(\tau)}\theta_{\G}^{\tau,T}&=\text{the minimum eigenvalue of }\Delta_{\G}(\rho_{\tau})\nonumber\\
		&\geq\begin{cases}
			 \frac{1}{6\log 3}\;r \theta_{\G}&\text{ for }r\in\{\lceil 6\log 3\rceil k,\lceil 6\log 3\rceil k+1,\dots,\lfloor \frac{n}{2}\rfloor\},\\[1ex]
			 \frac{r}{8\log n} \theta_{\G}&\text{ for }r\in\{\lfloor \frac{n}{2}\rfloor+1,\lfloor \frac{n}{2}\rfloor+2,\dots,n-1\},
		\end{cases}\nonumber\\
		&\geq\min\bigg\{ \frac{1}{6\log 3},\frac{\lfloor \frac{n}{2}\rfloor+1}{8 \lceil 6\log 3\rceil k \log n}\bigg\} \lceil 6\log 3\rceil k \theta_{\G}
	\end{align}
	We can choose positive integer $N:=N(k)$ (depending on $k$) such that
	\[\min\bigg\{ \frac{1}{6\log 3},\frac{\lfloor \frac{n}{2}\rfloor+1}{8 \lceil 6\log 3\rceil k\log n}\bigg\}=\frac{1}{6\log 3}\quad\text{for all}\quad n\geq N.\]
	Thus, for $n\geq N$, \eqref{eq:gen<->complete2} provides
	\[\displaystyle\min_{T\in\std(\tau)}\theta_{\G}^{\tau,T}=\text{the minimum eigenvalue of }\Delta_{\G}(\rho_{\tau})\geq \frac{\lceil 6\log 3\rceil }{6\log 3} k\theta_{\G}>k \theta_{\G}.\]
	Hence, for $\tau\vdash n$ satisfying $\tau_1<n-\lceil 6\log 3\rceil k+1$, we have $N:=N(k)$ such that
	\begin{equation}\label{eq:gen<->complete3}
		\text{every eigenvalue of }\Delta_{\G}(\rho_{\tau})>k \theta_{\G}\quad\text{for all}\quad n\geq N.
	\end{equation}
	Therefore, from \eqref{eq:gen<->complete1}, we have 
		\begin{align}\label{eq:gen<->comp_key}
		\Lambda_{\lceil 6\log 3\rceil k-1}^{\complement}\setminus\{(n)\}=&\{\tau\vdash n: \tau_1<n-\lceil 6\log 3\rceil k+1\}\nonumber\\
		\subseteq & \big\{\tau\vdash n: \text{all the eigenvalues of }\Delta_{\G}(\rho_{\tau})>k \theta_{\G}\big\}\quad\text{for all}\quad n\geq N.
	\end{align}
	The inclusion in \eqref{eq:gen<->comp_key} follows from \eqref{eq:gen<->complete3}.
	
	Now, we assume that $\{\f_{n}\}_{n=1}^{\infty}$ is noise sensitive with respect to the interchange process on $K_{n}$. Our aim is to use Lemma \ref{lem:spectral_NS} with the orthonormal eigenbasis given in \eqref{eq:torus<->complete-ONB}. For any positive integer $k$, and for all $n\geq N$, we obtain that
	\begin{align*}
		&\displaystyle\sum_{\substack{\lambda\vdash n\\\lambda\neq (n)}}\hspace*{1ex}\displaystyle\sum_{\substack{T\in\std(\lambda)\\
				\theta_{\G}\leq 	\theta^{\lambda,T}_{\G}< k\theta_{\G}}}\sum_{S\in\std(\lambda)}\left|\langle\f_n,\psi^{ST}\rangle\right|^2\nonumber\\
		\leq &\displaystyle\sum_{\lambda\in\Lambda_{\lceil 6\log 3\rceil k-1}\setminus\{(n)\}}\displaystyle\sum_{T\in\std(\lambda)}\sum_{S\in\std(\lambda)}\left|\langle\f_n,\psi^{ST}\rangle\right|^2,\nonumber\\
		&\qquad\quad\text{ because $\displaystyle\sum_{\lambda\in\Lambda^{\complement}_{\lceil 6\log 3\rceil k-1}\setminus\{(n)\}}\displaystyle\sum_{\substack{T\in\std(\lambda)\\
					\theta_{\G}\leq 	\theta^{\lambda,T}_{\G}< k\theta_{\G}}}$ is an empty sum, by \eqref{eq:gen<->comp_key}}\nonumber\\
		=&\displaystyle\sum_{\lambda\in\Lambda_{\lceil 6\log 3\rceil k-1}\setminus\{(n)\}}\sum_{S,T\in\std(\lambda)}\left|\langle\f_n,\psi^{ST}\rangle\right|^2
	\end{align*}
	But, the expression above converges to zero as $n\rightarrow \infty$, because $\{\f_{n}\}_{n=1}^{\infty}$ is noise sensitive with respect to the interchange process on $K_{n}$. Hence,
	\[\lim_{n\rightarrow\infty}\displaystyle\sum_{\substack{\lambda\vdash n\\\lambda\neq (n)}}\hspace*{1ex}\displaystyle\sum_{\substack{T\in\std(\lambda)\\
			\theta_{\G}\leq 	\theta^{\lambda,T}_{\G}< k\theta_{\G}}}\sum_{S\in\std(\lambda)}\left|\langle\f_n,\psi^{ST}\rangle\right|^2=0.\]
	Thus, the first part of the theorem follows from Lemma \ref{lem:spectral_NS}.

	Now, we assume that $\{\f_{n}\}_{n=1}^{\infty}$ is noise stable with respect to the interchange process on $\G$. For this case, we use Lemma \ref{lem:spectral_NStable} with the orthonormal eigenbasis given in \eqref{eq:torus<->complete-ONB}. For any positive integer $k$ and $n\geq N$, we obtain that
	\begin{align*}
		&\displaystyle\sum_{\substack{\lambda\vdash n\\  \lambda\in 	\Lambda_{\lceil 6\log 3\rceil k-1}^{\complement}\setminus\{(n)\}}}\sum_{S,T\in\std(\lambda)}\left|\langle\f_n,\psi^{ST}\rangle\right|^2\leq \displaystyle\sum_{\substack{\lambda\vdash n\\\lambda\neq (n)}}\hspace*{1ex}\displaystyle\sum_{\substack{T\in\std(\lambda)\\\theta^{\lambda,T}_{\G}\geq k \theta_{\G}}}\sum_{S\in\std(\lambda)}\left|\langle\f_n,\psi^{ST}\rangle\right|^2,\quad\text{by \eqref{eq:gen<->comp_key}}\nonumber
	\end{align*}
	But, the right hand expression above is bounded above by any arbitrary $\delta>0$ for some positive integer $k$, because $\{\f_{n}\}_{n=1}^{\infty}$ is noise stable with respect to the interchange process on $\G$. Hence,
	\begin{equation}\label{eq:gen<->comp_final}
		\sup_{n\geq N}\displaystyle\sum_{\substack{\lambda\vdash n\\  \lambda\in 	\Lambda_{\lceil 6\log 3\rceil k-1}^{\complement}\setminus\{(n)\}}}\sum_{S,T\in\std(\lambda)}\left|\langle\f_n,\psi^{ST}\rangle\right|^2<\delta.
	\end{equation}
	Once $k$ is fixed, the constant $N:=N(k)$ is determined. Thus, for $n\in\{1,\dots,N-1\}$, we have
	\[\text{Diag}(\lambda)> \binom{n}{2}-nN\qquad\text{for all}\qquad \lambda\vdash n.\]
	Therefore, using \eqref{eq:RT-partition_range}, we conclude that
	\[\Lambda_{N}^{\complement}\setminus\{(n)\}=\bigg\{\lambda\vdash n:\text{Diag}(\lambda)\leq \binom{n}{2}-nN\bigg\}=\emptyset,\quad\text{for}\quad 1\leq n <N.\]
	Hence, for $k':=\max\{N,\lceil 6\log 3\rceil k-1\}$, by \eqref{eq:gen<->comp_final}, we have
	\[\sup_{n\geq 1}\displaystyle\sum_{\substack{\lambda\vdash n\\  \lambda\in 	\Lambda_{k'}^{\complement}\setminus\{(n)\}}}\sum_{S,T\in\std(\lambda)}\left|\langle\f_n,\psi^{ST}\rangle\right|^2<\delta.\]
	Thus, the second part of the theorem follows from Lemma \ref{lem:spectral_NStable}.
\end{proof}
\section{Proof of Theorem \ref{thm:RT<-->s-cycle}}\label{sec:proof_of_thm:1.3} 
The interchange process on $K_n$ is the continuous-time random walk on the symmetric group $S_n$ generated by all transpositions. In this section, we consider the continuous-time random walk on $S_n$ generated by all $s$-cycles ($s\ll n$ is an even), and prove Theorem \ref{thm:RT<-->s-cycle}. We denote the conjugacy class of all $s$-cycles of $S_n$ by the notation $\mathscr{C}_s$. For our continuous-time random walk on $S_n$, the associated rate function $\mu_{s}:S_n\rightarrow (0,\infty)$ is defined by
\[\mu_{s}(\pi)=\begin{cases}
	\nu''&\text{ if }\pi\in \mathscr{C}_s\\
	0&\text{ otherwise.}\
\end{cases}\]
and the infinitesimal generator is given by the matrix of the operator
\[-\Delta_{s}(R):=\nu''\displaystyle\sum_{\pi\in \mathscr{C}_s}(R_{\pi}-I)\]
on $\mathbb{C}^{S_n}$ with respect to the basis $\{\mathbb{1}_{\{\pi\}}:\pi\in S_n\}$. Here, $R$ is the regular representation of $S_n$ and $I$ is the identity operator. Using arguments similar to those in Lemma \ref{lem:compgrp_eigval_on_irr}, we obtain
 \begin{lem}\label{lem:s-cyc_eigval_on_irr}
 	For any $\lambda\vdash n$, let $(\rho_{\lambda},V^{\lambda})$ denote the irreducible representation of $S_n$ indexed by $\lambda$, we denote the corresponding irreducible character by $\chi^{\lambda}$. Then,
 	\[
 	\Delta_{s}(\rho_{\lambda}):=\nu''\displaystyle\sum_{\pi\in \mathscr{C}_s}(I-\rho_{\lambda}(\pi))=\nu''|\mathscr{C}_s|\left(1-\frac{\chi^{\lambda}(\sigma)}{f^{\lambda}}\right) I,\quad \text{i.e., }\displaystyle\sum_{\pi\in \mathscr{C}_s}\rho_{\lambda}(\pi)=|\mathscr{C}_s|\frac{\chi^{\lambda}(\sigma)}{f^{\lambda}} I.
 	\]
 \end{lem}
 Again, using Lemma \ref{lem:s-cyc_eigval_on_irr} and proceeding as in the proof of Lemma \ref{lem:compgrp_eigval}, we obtain
 \begin{lem}\label{lem:s-cyc_eigval}
 	Let $\lambda\vdash n$ and $S,T\in\std(\lambda)$. For any orthonormal basis $\mathcal{B}_{\lambda}$ of $V^{\lambda}$, recall that $\psi^{ST}_{\mathcal{B}_{\lambda}}(\pi):=\sqrt{f^{\lambda}}\Tr\left(\big[\rho_{\lambda}(\pi^{-1})\big]_{\mathcal{B}_{\lambda}}E^{ST}_{\mathcal{B}_{\lambda}}\right)$ for all $\pi\in S_n$. Then, the orthonormal basis 
 	\[	\underset{\lambda\vdash n}{\bigcup}\big\{\psi^{ST}_{\mathcal{B}_{\lambda}}:S,T\in\std(\lambda)\big\}\]
 	given in \eqref{eq:ONB_for_Sn} is an eigenbasis of $\Delta_{s}(R)$. Moreover, the eigenvalue associated to the eigenvector $\psi^{ST}_{\mathcal{B}_{\lambda}}$ is $\nu''|\mathscr{C}_s|\left(1-\frac{\chi^{\lambda}(\sigma)}{f^{\lambda}}\right),\;S,T\in\std(\lambda)$ for every $\lambda\vdash n$.
 \end{lem}
The smallest eigenvalue is $0$, comes for the index $\lambda=(n)$. For sufficiently large $n$, the first positive eigenvalue is $\nu''|\mathscr{C}_s|\left(1-\frac{n-s-1}{n-1}\right)$ for the index $\lambda=(n-1,1)$ (\cite[Theorem 1.2]{LXZ,MR-second_evalue}. For the other eigenvalues, we recall an estimates of the $S_n$ characters from \cite{N-limit} below.
\begin{prop}{\cite[Corollary 3.4]{N-limit}}\label{prop:s-cycle_ch._estimate}
	Let $\lambda=(\lambda_1,\dots,\lambda_{\ell})$ be a partition of $n$ with long first row, i.e., $\lambda_1=n-r$ for an integer $r$ satisfying $r+s+1<\frac{1}{3}n$. Then for large enough $n$
	\[	\frac{\chi^{\lambda}(\sigma)}{f^{\lambda}}=e^{-\frac{rs}{n}}\left(1+O\left(\frac{s}{n^2}\right)\right),\;\text{ for all }\;2\leq s\leq\frac{1}{3}n.\]
\end{prop}
Now, for any positive $k$, we have
\[\nu''\left|\mathscr{C}_s\right|\left(1-\frac{n-s-1}{n-1}\right)\leq\nu''|\mathscr{C}_s|\left(1-\frac{\chi^{\lambda}(\sigma)}{f^{\lambda}}\right)<k\nu''\left|\mathscr{C}_s\right|\left(1-\frac{n-s-1}{n-1}\right),\]
and set
\begin{equation}\label{eq:s-cycle_partition_range}
	\Theta_k:=\Bigg\{\lambda\vdash n: 1-\frac{ks}{n-1}<\frac{\chi^{\lambda}(\sigma)}{f^{\lambda}}\leq\frac{n-s-1}{n-1}\Bigg\}.
\end{equation}
Therefore, Lemma \ref{lem:compgrp_boolfn_projec}, Lemma \ref{lem:spectral_NS}, and Lemma \ref{lem:spectral_NStable} implies the following theorem:
\begin{thm}\label{thm:s-cyc_NS-NStability}
	Let $\lambda\vdash n$, and $\mathcal{B}_{\lambda}$ be any orthonormal basis of $V^{\lambda}$. Also, for every $n\geq 1$, let $\f_n:S_n\longrightarrow\{0,1\}$ be a Boolean function. Then,
	\begin{enumerate}
		\item The sequence of $\{\f_n\}_n$ is noise sensitive with respect to the continuous-time random walk on $S_n$ generated by all $s$-cycles ($s\ll n$ is even) if and only if for all positive integer $k>0$
		\[	\lim_{n\rightarrow\infty}\;\sum_{\lambda\in\Theta_k} \frac{f^{\lambda}}{n!^2}\sum_{S,T\in\std(\lambda)}\left|\text{The }(S,T)\text{th entry of the matrix }\Big[\widehat{\f_n}(\lambda)\Big]_{\mathcal{B}_{\lambda}}\right|^2=0.\]
		\item The sequence of $\{\f_n\}_n$ is noise stable with respect to the continuous-time random walk on $S_n$ generated by all $s$-cycles ($s\ll n$ is even) if and only if for all $\delta>0$ there exists a positive integer $k$ such that
		\[\sup_n\;\sum_{\substack{\lambda\vdash n\\
				\lambda\notin \Theta_k\cup\{(n)\}}} \frac{f^{\lambda}}{n!^2}\sum_{S,T\in\std(\lambda)}\left|\text{The }(S,T)\text{th entry of the matrix }\Big[\widehat{\f_n}(\lambda)\Big]_{\mathcal{B}_{\lambda}}\right|^2<\delta.\]
	\end{enumerate}
\end{thm}
\begin{proof}[Proof of Theorem \ref{thm:RT<-->s-cycle}]
 	Let $k$ be any arbitrary positive integer. We first show that $\Theta_k$ is either $\Lambda_k$ or $\Lambda_{k-1}$, for sufficiently large $n$. We note that $2<s\ll n$, and $\lambda=(\lambda_1,\dots,\lambda_{\ell})\vdash n$ satisfies $r=n-\lambda_1\leq k+1<\frac{1}{3}n-s-1$. Therefore, Proposition \ref{prop:s-cycle_ch._estimate} implies 
 	\[\frac{\chi^{\lambda}(\sigma)}{f^{\lambda}}=e^{-\frac{rs}{n}}\left(1+O\left(\frac{s}{n^2}\right)\right)=1-\frac{rs}{n}+O\left(\frac{s^2r^2}{n^2}\right)=1-\frac{rs}{n-1}+O\left(\frac{s^2r^2}{n^2}\right).\]
 	Thus, $\lambda\in\Theta_{k}$ if and only if, $1-\frac{ks}{n-1}<\frac{\chi^{\lambda}(\sigma)}{f^{\lambda}}\leq 1-\frac{s}{n-1}$, i.e., if and only if,
 	\begin{align*}
 	&\;1-\frac{ks}{n-1}<1-\frac{rs}{n-1}+O\left(\frac{s^2r^2}{n^2}\right)\leq 1-\frac{s}{n-1}\text{ or }\lambda=(n-1,1),\\
 	&\text{i.e., if and only if }\;\;1+O\left(\frac{sr^2}{n}\right)\leq r<k+O\left(\frac{sr^2}{n}\right)\text{ or }r=1,
 	\end{align*}
 	i.e., if and only if, either $1\leq r\leq k-1$ or $1\leq r\leq k$. Therefore for sufficiently large $n$, \eqref{eq:RT-partition_range-decomposition} implies that $\Theta_k$ is either $\Lambda_k$ or $\Lambda_{k-1}$. Hence the theorem follows from Theorem \ref{thm:compgrp_NS-NStability} and Theorem \ref{thm:s-cyc_NS-NStability}.
 \end{proof}
\section{Proof of Theorem \ref{thm:large_cycle1}}\label{sec:proof_of_thm:1.4}
In this section we establish the Theorem \ref{thm:large_cycle1}. Unless otherwise stated, we work with the normalized GT-basis $\mathscr{B}_{\lambda},\lambda\vdash n$ given in Subsection \ref{subsec:star_graph}. We start with proving the following useful theorem.
\begin{thm}\label{thm:cycle-length_sum}
	Let $0<c<1$ be fixed, and let $\alpha_u:S_n\rightarrow [n]$ be the function defined by
	\[\alpha_u(\pi):=\text{the number of u-cycles in }\pi,\;\;\pi\in S_n\text{ and }1\leq u\leq n.\]
	Then, for every fixed positive integer $k$, we have,
	\[\sum_{\lambda\in \Lambda_k} \frac{f^{\lambda}}{n!^2}\sum_{S,T\in\std(\lambda)}\left|\text{The }(S,T)\text{th entry of the matrix }\left[\widehat{\displaystyle\sum_{j\geq cn}\alpha_j}(\rho_{\lambda})\right]\right|^2<\frac{k^2}{(n-k+1)^2},\]
	for all $n\geq 1$.
\end{thm}
\begin{proof}
	We first set the notation
	\begin{equation}\label{eq:large_cycle1.1}
		\mathfrak{a}_n:=\displaystyle\sum_{j\geq cn}\alpha_j.
	\end{equation}
	Let $(\rho_{\lambda},V^{\lambda})$ denote the irreducible representation of $S_n$ indexed by $\lambda\vdash n$; we denote the corresponding irreducible character using notation $\chi^{\lambda}$. For any $n$ and $j\geq cn$, Theorem 3 of \cite{AK} implies that
	\begin{equation}\label{eq:large_cycle1.2}
		\alpha_j=\frac{1}{j}\left(\chi^{(n)}+\sum_{0\leq i\leq 2j-n-2}(-1)^{i+1}\chi^{(j-i-1,n-j+1,1^i)}+\sum_{2j-n\leq i\leq j-1}(-1)^{i}\chi^{(n-j,j-i,1^i)}\right).
	\end{equation}
	As the function $\mathfrak{a}_n$ is constant on conjugacy classes, Schur's lemma implies that 
	\begin{equation}\label{eq:large_cycle1.3}
		\widehat{\mathfrak{a}}_n\left(\rho_{\lambda}\right):=\eta_{\lambda} [I]_{\lambda},\text{ where }\eta_{\lambda}f^{\lambda}=\sum_{\pi\in S_n}\mathfrak{a}_n(\pi)\chi^{\lambda}(\pi)=\sum_{\pi\in S_n}\chi^{\lambda}(\pi)\overline{\mathfrak{a}_n(\pi)}=n!\langle \chi^{\lambda},\mathfrak{a}_n\rangle,
	\end{equation}
	for all $\lambda\vdash n$. Here $[I]_{\lambda}$ is the identity matrix of size $f^{\lambda}\times f^{\lambda}$.
	Now, using \eqref{eq:large_cycle1.1}, we have
	\[\eta_{\lambda}=\frac{n!}{f^{\lambda}}\langle\chi^{\lambda},\mathfrak{a}_n\rangle=\frac{n!}{f^{\lambda}}\sum_{j\geq cn}\langle\chi^{\lambda},\alpha_j\rangle.\]
	Thus, the orthonormality of the irreducible characters, and \eqref{eq:large_cycle1.2} implies that $\eta_{\lambda}$ is equal to
	\begin{equation}\label{eq:large_cycle1.4}
		\frac{n!}{f^{\lambda}}\sum_{j\geq cn}\frac{1}{j}\left(\delta_{\lambda,(n)}+\hspace*{-0.25cm}\sum_{0\leq i\leq 2j-n-2}(-1)^{i+1}\delta_{\lambda,(j-i-1,n-j+1,1^i)}+\hspace*{-0.25cm}\sum_{2j-n\leq i\leq j-1}(-1)^{i}\delta_{\lambda,(n-j,j-i,1^i)}\right),
	\end{equation}
	for all $\lambda\vdash n$. Let us now choose any arbitrary positive integer $k$. Then, for  sufficiently large $n$, using \eqref{eq:RT-partition_range-decomposition}, we have
	\[\Lambda_k=\bigcup_{r=1}^{k}\Lambda_{k,r},\quad
	\text{ where }\Lambda_{k,r}=\bigg\{(n-r,\mu)\vdash n:\mu\vdash r\bigg\}.\]
	For $1\leq r\leq k$ and $\lambda:=(n-r,\mu)\in\Lambda_{k,r}$, we have 
	\[\lambda\notin\bigg\{(n)\bigg\}\bigcup\left(\underset{j\geq cn}{\bigcup}\quad\underset{2j-n\leq i\leq j-1}{\bigcup}\bigg\{(n-j,j-i,1^i)\bigg\}\right),\]
	hence, using \eqref{eq:large_cycle1.4}, we have
	\begin{align}
		\eta_{\lambda}&=\frac{n!}{f^{\lambda}}\sum_{j\geq cn}\frac{1}{j}\sum_{0\leq i\leq 2j-n-2}(-1)^{i+1}\delta_{\lambda,(j-i-1,n-j+1,1^i)}\label{eq:large_cycle1.5}\\
		&=\frac{n!}{f^{\lambda}}\sum_{j\geq n-r+1}\frac{1}{j}\sum_{0\leq i\leq 2j-n-2}(-1)^{i+1}\delta_{\lambda,(j-i-1,n-j+1,1^i)}\nonumber\\
		&=\frac{n!}{f^{\lambda}}\sum_{j\geq n-r+1}\frac{(-1)^{j-n+r}}{j}\delta_{\lambda,(n-r,n-j+1,1^{j-n+r-1})}\nonumber\\
		&=\frac{n!}{f^{\lambda}}\sum_{j\geq n-r+1}\frac{(-1)^{j-n+r}}{j}\delta_{\mu,(n-j+1,1^{j-n+r-1})}=\frac{n!}{f^{\lambda}}\sum_{v=1}^{r}\frac{(-1)^{r-v+1}}{n-v+1}\delta_{\mu,(v,1^{r-v})}\nonumber
	\end{align}
	Therefore, using \eqref{eq:large_cycle1.3}, we have
	\begin{align}
		&\sum_{\lambda\in \Lambda_k} \frac{f^{\lambda}}{n!^2}\sum_{S,T\in\std(\lambda)}\left|\text{The }(S,T)\text{th entry of the matrix }\Big[\widehat{\mathfrak{a}}_n(\rho_{\lambda})\Big]\right|^2\nonumber\\
		=&\sum_{\lambda\in \Lambda_k} \frac{f^{\lambda}}{n!^2}\;f^{\lambda}\;|\eta_{\lambda}|^2\nonumber\\
		=&\sum_{r=1}^{k}\quad\sum_{\lambda=(n-r,\mu)\in \Lambda_{k,r}} \left|\frac{f^{\lambda}}{n!}\eta_{\lambda}\right|^2\nonumber\\
		=&\sum_{r=1}^{k}\quad\sum_{\lambda=(n-r,\mu)\in \Lambda_{k,r}} \left|\sum_{v=1}^{r}\frac{(-1)^{r-v+1}}{n-v+1}\delta_{\mu,(v,1^{r-v})}\right|^2\label{eq:large_cycle1.6}
	\end{align}
	The last equality in \eqref{eq:large_cycle1.6} follows from \eqref{eq:large_cycle1.5}. Thus, we obtain,
	\begin{align}
		&\sum_{\lambda\in \Lambda_k} \frac{f^{\lambda}}{n!^2}\sum_{S,T\in\std(\lambda)}\left|\text{The }(S,T)\text{th entry of the matrix }\left[\widehat{\displaystyle\sum_{j\geq cn}\alpha_j}(\rho_{\lambda})\right]\right|^2\nonumber\\
		=&\sum_{r=1}^{k}\quad\sum_{\mu\vdash r} \left(\sum_{v=1}^{r}\frac{(-1)^{r-v+1}}{n-v+1}\delta_{\mu,(v,1^{r-v})}\right)^2\nonumber\\
		=&\sum_{r=1}^{k}\sum_{t=1}^{r} \left(\sum_{v=1}^{r}\frac{(-1)^{r-v+1}}{n-v+1}\delta_{(t,1^{r-t}),(v,1^{r-v})}\right)^2\nonumber\\
		=&\sum_{r=1}^{k}\sum_{t=1}^{r} \left(\frac{(-1)^{r-t+1}}{n-t+1}\right)^2=\sum_{r=1}^{k}\;\sum_{t=n-r+1}^{n} \frac{1}{t^2}=\sum_{r=0}^{k-1}\frac{k-r}{(n-r)^2}.\label{eq:large_cycle1.7}
	\end{align}
	Finally, the theorem follows from $\displaystyle\sum_{r=0}^{k-1}\frac{k-r}{(n-r)^2}\leq \displaystyle\sum_{r=0}^{k-1} \frac{k}{(n-r)^2}<\frac{k^2}{(n-k+1)^2}$.
\end{proof}
\begin{cor}\label{cor:cycle-length>(n/2)}
	Let $\frac{1}{2}< c<1$ be fixed. Then, the function $\phi_n:S_n\rightarrow \{0,1\}$ defined by
	\[\phi_n:=\mathbb{1}_{\{\pi\in S_n:\text{the length of the largest cycle in }\pi\text{ is at least }cn\}},\]
	is noise sensitive with respect to the interchange process on $K_n$.
\end{cor}
\begin{proof}
	We recall the function $\alpha_u:S_n\rightarrow\{1,\dots,n\}$ from Theorem \ref{thm:cycle-length_sum}. Using the fact $c>\frac{1}{2}$, we can conclude that $\alpha_j$ is a Boolean function for all $j\geq cn$. Moreover, the support of $\alpha_i$ and $\alpha_j$ are disjoint for all distinct $i,j\geq cn$. Thus, $\displaystyle\sum_{j\geq cn}\alpha_j$ is also a Boolean function, and 
	\[\phi_n=\displaystyle\sum_{j\geq cn}\alpha_j.\]
	Now for any fixed positive integer $k$, using $\displaystyle\lim_{n\rightarrow\infty}\frac{k^2}{(n-k+1)^2}=0$ and Theorem \ref{thm:cycle-length_sum}, the corollary follows from the first part of Theorem \ref{thm:compgrp_NS-NStability}.
\end{proof}
\begin{cor}\label{cor:cycle-length_geq_(n/2)}
	Let $\frac{1}{2}\leq c<1$ be fixed. Then, the function $\phi_n':S_n\rightarrow \{0,1\}$ defined by
	\[\phi_n':=\mathbb{1}_{\{\pi\in S_n:\text{the length of the largest cycle in }\pi\text{ is at least }cn\}},\]
	is noise sensitive with respect to the interchange process on $K_n$.
\end{cor}
\begin{proof}
	We can write $\phi_n'$ as follows
	\[\phi_n'=\displaystyle\sum_{j\geq cn}\alpha_j-\mathbb{1}_{\{\pi\in S_n:\text{the cycle type of }\pi\text{ is }(\frac{n}{2},\frac{n}{2})\}}.\]
	For every $\lambda\vdash n$, let $(\rho_{\lambda},V^{\lambda})$ denote the irreducible representation of $S_n$ indexed by $\lambda$; we denote the corresponding irreducible character using notation $\chi^{\lambda}$.
	Therefore, using the Schur's lemma, and the notations from \eqref{eq:large_cycle1.3}, we have 
	\[\widehat{\phi'}_n\left(\rho_{\lambda}\right)=\left(\eta_{\lambda}-\frac{2n!}{n^2}\frac{\chi^{\lambda}(\frac{n}{2},\frac{n}{2})}{f^{\lambda}}\right) [I]_{\lambda}.\]
	Now for any fixed positive integer $k$, using Theorem \ref{thm:cycle-length_sum}, from \eqref{eq:large_cycle1.3}, we obtain
	\begin{align}
		&\sum_{\lambda\in \Lambda_k} \frac{f^{\lambda}}{n!^2}\sum_{S,T\in\std(\lambda)}\left|\text{The }(S,T)\text{th entry of the matrix }\Big[\widehat{\displaystyle\sum_{j\geq cn}\alpha_j}(\rho_{\lambda})\Big]\right|^2\nonumber\\
		=&\sum_{\lambda\in \Lambda_k} \left|\frac{f^{\lambda}}{n!}\eta_{\lambda}\right|^2\leq\frac{k^2}{(n-k+1)^2}.\label{eq:long_cycle1.1}
	\end{align}
	Therefore,
	\begin{align}
		&\sum_{\lambda\in \Lambda_k} \frac{f^{\lambda}}{n!^2}\sum_{S,T\in\std(\lambda)}\left|\text{The }(S,T)\text{th entry of the matrix }\Big[\widehat{\phi'}_n(\rho_{\lambda})\Big]\right|^2\nonumber\\
		=&\sum_{\lambda\in \Lambda_k} \frac{\left(f^{\lambda}\right)^2}{n!^2}\left|\eta_{\lambda}-\frac{2n!}{n^2}\frac{\chi^{\lambda}(\frac{n}{2},\frac{n}{2})}{f^{\lambda}}\right|^2\nonumber\\
		=&\sum_{\lambda\in \Lambda_k}\left|\frac{f^{\lambda}}{n!}\eta_{\lambda}-\frac{2\chi^{\lambda}(\frac{n}{2},\frac{n}{2})}{n^2}\right|^2\nonumber\\
		\leq &\sum_{\lambda\in \Lambda_k}2\left(\left|\frac{f^{\lambda}}{n!}\eta_{\lambda}\right|^2+\left|\frac{2\chi^{\lambda}(\frac{n}{2},\frac{n}{2})}{n^2}\right|^2\right)\nonumber\\
		\leq& \frac{2k^2}{(n-k+1)^2}+2\left(\frac{4}{n^2}\right)^2\sum_{r=1}^{k}r!\label{eq:long_cycle1.2}
	\end{align}
	The inequality in \eqref{eq:long_cycle1.2} follows from \eqref{eq:long_cycle1.1}, \eqref{eq:RT-partition_range-decomposition}, $|\Lambda_{k,r}|\leq r!$, and 
	\[\left|\chi^{\lambda}\left(\frac{n}{2},\frac{n}{2}\right)\right|\leq2\text{ for all }\lambda\in\Lambda_k\;\text{ (by \emph{Murnaghan-Nakayama} rule)}.\]
	Therefore, the theorem follows from the first part of Theorem \ref{thm:compgrp_NS-NStability}, and the fact that the right hand side of \eqref{eq:long_cycle1.2} goes to zero as $n\rightarrow\infty$.
\end{proof}
The aforementioned two corollaries illustrate the roadmap for the proof of Theorem \ref{thm:large_cycle1}.
\begin{proof}[Proof of Theorem \ref{thm:large_cycle1}]
	We first recall that the \emph{cycle type} of a permutation in $S_n$ is the partition of $n$ whose parts are precisely the lengths of the cycles in its cycle decomposition. We now introduce the following notation
	\begin{align*}
		\alpha_u(\pi)&:=\text{the number of }u\text{-cycles in }\pi,\;\;\pi\in S_n\text{ and }1\leq u\leq n,\\
		\xi_n&:=\mathbb{1}_{\{\pi\in S_n:\text{the length of the largest cycle in }\pi\text{ is at least }cn\}},\\
		M_{\mu}&:=\left|\left\{i\in [\ell]:\mu_i\geq cn\right\}\right|\leq\frac{1}{c}\;\text{ for all }\mu=(\mu_1,\mu_2,\dots,\mu_{\ell})\vdash n,\\
		\beta_{\mu}&:=\mathbb{1}_{\{\pi\in S_n:\text{the cycle type of }\pi\text{ is }\mu\}}\;\text{ for all }\mu\vdash n,\\
		\mathscr{C}_{\mu}&:=\{\pi\in S_n:\text{the cycle type of }\pi\text{ is }\mu\}\;\text{ for all }\mu\vdash n.
	\end{align*}
	Given any $\pi\in S_n$ whose longest cycle has length at least $cn$, we can observe that
	\[M_{\text{cycle type of }\pi}=\sum_{j\geq cn}\alpha_j.\]
	Therefore, we have
	\[\xi_n=\sum_{j\geq cn}\alpha_j-\sum_{\mu\vdash n:\;M_{\mu}\geq 2}(M_{\mu}-1)\beta_{\mu}.\]
	For every $\lambda\vdash n$, let $(\rho_{\lambda},V^{\lambda})$ denote the irreducible representation of $S_n$ indexed by $\lambda$; we denote the corresponding irreducible character using notation $\chi^{\lambda}$.
	Therefore, by Schur's lemma, and the notations from \eqref{eq:large_cycle1.3}, we have 
	\begin{equation}\label{eq:long_cycle-final1.1}
		\widehat{\xi}_n\left(\rho_{\lambda}\right)=\left(\eta_{\lambda}-\sum_{\mu\vdash n:\;M_{\mu}\geq 2}(M_{\mu}-1)|\mathscr{C}_{\mu}|\frac{\chi^{\lambda}(\mu)}{f^{\lambda}}\right) [I]_{\lambda}.
	\end{equation}
	Here, $\chi^{\lambda}(\mu)$ denotes the value of the irreducible character $\chi^{\lambda}$ at any permutation of cycle type $\mu$, and $[I]_{\lambda}$ denotes the identity matrix of size $f^{\lambda}\times f^{\lambda}$. Now for any fixed positive integer $k$, using Theorem \ref{thm:cycle-length_sum}, from \eqref{eq:large_cycle1.3}, we obtain
	\begin{align}
		&\sum_{\lambda\in \Lambda_k} \frac{f^{\lambda}}{n!^2}\sum_{S,T\in\std(\lambda)}\left|\text{The }(S,T)\text{th entry of the matrix }\Big[\widehat{\displaystyle\sum_{j\geq cn}\alpha_j}(\rho_{\lambda})\Big]\right|^2\nonumber\\
		=&\sum_{\lambda\in \Lambda_k} \left|\frac{f^{\lambda}}{n!}\eta_{\lambda}\right|^2\leq\frac{k^2}{(n-k+1)^2}.\label{eq:long_cycle-final1.2}
	\end{align}
	Therefore, we have
	\begin{align}
		&\sum_{\lambda\in \Lambda_k} \frac{f^{\lambda}}{n!^2}\sum_{S,T\in\std(\lambda)}\left|\text{The }(S,T)\text{th entry of the matrix }\Big[\widehat{\xi}_n(\rho_{\lambda})\Big]\right|^2\nonumber\\
		=&\sum_{\lambda\in \Lambda_k} \frac{\left(f^{\lambda}\right)^2}{n!^2}\left|\eta_{\lambda}-\sum_{\mu\vdash n:\;M_{\mu}\geq 2}(M_{\mu}-1)|\mathscr{C}_{\mu}|\frac{\chi^{\lambda}(\mu)}{f^{\lambda}}\right|^2,\text{ by \eqref{eq:long_cycle-final1.1}}\nonumber\\
		=&\sum_{\lambda\in \Lambda_k}\left|\frac{f^{\lambda}}{n!}\eta_{\lambda}-\sum_{\mu\vdash n:\;M_{\mu}\geq 2}(M_{\mu}-1)|\mathscr{C}_{\mu}|\frac{\chi^{\lambda}(\mu)}{n!}\right|^2\nonumber\\
		\leq &\sum_{\lambda\in \Lambda_k}\left(\left|\frac{f^{\lambda}}{n!}\eta_{\lambda}\right|+\sum_{\mu\vdash n:\;M_{\mu}\geq 2}\left|(M_{\mu}-1)|\mathscr{C}_{\mu}|\frac{\chi^{\lambda}(\mu)}{n!}\right|\right)^2\nonumber\\
		\leq &\sum_{\lambda\in \Lambda_k}2\left(\left|\frac{f^{\lambda}}{n!}\eta_{\lambda}\right|^2+\sum_{\mu\vdash n:\;M_{\mu}\geq 2}\left|(M_{\mu}-1)|\mathscr{C}_{\mu}|\frac{\chi^{\lambda}(\mu)}{n!}\right|^2\right)\nonumber\\
		\leq& \frac{2k^2}{(n-k+1)^2}+2\left(\frac{1}{c}-1\right)^2\sum_{\lambda\in \Lambda_k}\sum_{\mu\vdash n:\;M_{\mu}\geq 2}|\mathscr{C}_{\mu}|^2\left|\frac{\chi^{\lambda}(\mu)}{n!}\right|^2\label{eq:long_cycle-final1.3}
	\end{align}
	The inequality in \eqref{eq:long_cycle-final1.3} follows from \eqref{eq:long_cycle-final1.2} and $2\leq M_{\mu}\leq\frac{1}{c}$. Moreover, $M_{\mu}\geq 2$ implies that $\frac{|\mathscr{C}_{\mu}|}{n!}\leq\frac{1}{(cn)^2}$ (see the formula in \cite[eq. (1,2)]{Sagan}). Thus,
	\begin{align}
		\sum_{\mu\vdash n:\;M_{\mu}\geq 2}|\mathscr{C}_{\mu}|^2\left|\frac{\chi^{\lambda}(\mu)}{n!}\right|^2\leq& \frac{1}{(cn)^2}\frac{1}{n!}\sum_{\mu\vdash n:\;M_{\mu}\geq 2}|\mathscr{C}_{\mu}|\left|\chi^{\lambda}(\mu)\right|^2\nonumber\\
		\leq&\frac{1}{(cn)^2}\frac{1}{n!}\sum_{\mu\vdash n}|\mathscr{C}_{\mu}|\left|\chi^{\lambda}(\mu)\right|^2\nonumber\\
		=&\frac{1}{(cn)^2}\frac{1}{n!}\sum_{\pi\in \mathscr{C}_{\mu}}\left|\chi^{\lambda}(\pi)\right|^2=\frac{1}{(cn)^2}\langle\chi^{\lambda},\chi^{\lambda}\rangle=\frac{1}{(cn)^2}.\label{eq:long_cycle-final1.4}
	\end{align}
	The last equality in \eqref{eq:long_cycle-final1.4} follows from the orthonormality of irreducible characters. Finally,
	\begin{align}
		&\sum_{\lambda\in \Lambda_k} \frac{f^{\lambda}}{n!^2}\sum_{S,T\in\std(\lambda)}\left|\text{The }(S,T)\text{th entry of the matrix }\Big[\widehat{\xi}_n(\rho_{\lambda})\Big]\right|^2\nonumber\\
		\leq& \frac{2k^2}{(n-k+1)^2}+2\left(\frac{1}{c}-1\right)^2\sum_{\lambda\in \Lambda_k}\sum_{\mu\vdash n:\;M_{\mu}\geq 2}|\mathscr{C}_{\mu}|^2\left|\frac{\chi^{\lambda}(\mu)}{n!}\right|^2,\text{ by \eqref{eq:long_cycle-final1.3}}\nonumber\\
		\leq& \frac{2k^2}{(n-k+1)^2}+2\left(\frac{1}{c}-1\right)^2\sum_{\lambda\in \Lambda_k}\frac{1}{(cn)^2},\text{ by \eqref{eq:long_cycle-final1.4}}\nonumber\\
		\leq& \frac{2k^2}{(n-k+1)^2}+2\left(\frac{1}{c}-1\right)^2\frac{1}{(cn)^2}\sum_{r=1}^{k}r!.\label{eq:long_cycle-final1.5}
	\end{align}
	The inequality in \eqref{eq:long_cycle-final1.5} follows from \eqref{eq:RT-partition_range-decomposition} and $|\Lambda_{k,r}|\leq r!$. Therefore, the theorem follows from the first part of Theorem \ref{thm:compgrp_NS-NStability}, and the fact that the right hand side of \eqref{eq:long_cycle-final1.5} goes to zero as $n\rightarrow\infty$.
\end{proof}
%----------REFERENCES---------
\bibliography{ref}{}
\bibliographystyle{plain}%{alpha}
\end{document}